\documentclass[11pt, a4paper, reqno]{amsart}

\usepackage{pdflscape} 
\usepackage{hhline} 
\usepackage{array} 

\usepackage{calrsfs}
\DeclareMathAlphabet{\pazocal}{OMS}{zplm}{m}{n}
\usepackage[margin=0.9in,footskip=0.5in]{geometry}

\usepackage{latexsym,amsmath,amsfonts,amscd,amssymb, mathrsfs, amsthm}
\usepackage[english]{babel}
\usepackage{appendix}
\usepackage{hyperref}
\usepackage{mathtools}
\usepackage{booktabs}

\theoremstyle{plain} 
\newtheorem{theorem}{Theorem}[section]

\newtheorem*{theorem*}{Theorem}

\newtheorem{corollary}[theorem]{Corollary}
\newtheorem{lemma}[theorem]{Lemma}
\newtheorem{fact}[theorem]{Fact}

\theoremstyle{definition}
\newtheorem{definition}[theorem]{Definition}

\newtheorem{remark}[theorem]{Remark}

\newtheorem*{claim*}{Claim}

\numberwithin{equation}{section}

\def\R{\mathbb{R}}

\def\Z{\mathbb{Z}}

\def\Aff{\operatorname{Aff}}

\def\Aut{\operatorname{Aut}}

\def\F{{\mathcal{F}}}
\def\P{\operatorname{P}}

\def\SOL{{\operatorname{Sol}} }

\def\Isom{{\operatorname{Isom}}}

\def\ZZ{{\operatorname{Z}} }

\def\heis{{\mathfrak{heis}}}

\def\s{{\mathfrak{s}}}

\def\Aut{{\operatorname{Aut}}}

\def\Heis{{\operatorname{Heis}}}

\def\PSL{{\operatorname{PSL}}}
\def\GL{{\operatorname{GL}}}
\def\O{{\operatorname{O}}}
\def\SO{{\operatorname{SO}}}
\def\SU{{\operatorname{SU}}}
 
\def\SL{{\operatorname{SL}}}

\def\Lie{{\operatorname{Lie}}}

\def\Span{{\operatorname{span}}}

\def\tr{{\operatorname{tr}}}

\def\det{{\operatorname{det}}}

\def\l{{\mathfrak{l}}}

\def\g{{\mathfrak{g}}}

\def\so{{\mathfrak{so}}}

\def\s{{\mathfrak{s}}}

\def\i{{\mathfrak{i}}}

\newcommand{\ad}{\mathrm{ad}}
\newcommand{\Ad}{\mathrm{Ad}}

\usepackage{xcolor}
\definecolor{MyBlue}{RGB}{0,0,255}

\definecolor{MyRed}{RGB}{255,0,0}

\definecolor{MyGray}{RGB}{150,60,60}

\newcommand{\be}{\begin{equation}}
	\newcommand{\ee}{\end{equation}}

\newcommand{\ben}{\begin{enumerate}}
	\newcommand{\een}{\end{enumerate}}
\newcommand{\bit}{\begin{itemize}}
	\newcommand{\eit}{\end{itemize}}
\newcommand{\edoc}{\end{document}}

\def\br#1\er{{#1}} 
\def\bw#1\ew{\textcolor{brown}{#1}} 
\def\bb#1\eb{\textcolor{blue}{#1}} 
\def\br#1\er{\textcolor{red}{#1}} 
\def\bm#1\em{\textcolor{magenta}{#1}}
\def\bv#1\ev{\textcolor{olive}{#1}}

\newtheorem{thm}{Theorem}[section]

\newtheorem{proposition}[thm]{Proposition}

\makeatletter
\usepackage{csquotes}

\usepackage[style=alphabetic,backend=biber, maxbibnames=99]{biblatex}
\renewbibmacro{in:}{}
\DeclareFieldFormat[article,inbook]{title}{#1} 
\def\Lie{{\operatorname{Lie}}}

\author{S. Allout} 
\address{Faculty of mathematics, Ruhr-Universitaet Bochum, Germany}
\email{souheib.allout@rub.de}%

\author{A. Belkacem}
\address{Department of mathematics, University of Batna 2, Algeria}
 \email{abderrahmane.belkacem.matea@gmail.com}

\author{A. Zeghib}
\address{UMPA, CNRS, ENS de Lyon, France}
  \email{abdelghani.zeghib@ens-lyon.fr}

\begin{document}

\title{\textbf{On homogeneous 3-dimensional spacetimes: focus on plane waves}}

\date{}

 \begin{abstract} We revisit the classification of Lorentz homogeneous spaces of dimension $3$, and relax usual completeness assumptions. In particular, non-unimodular elliptic plane waves, and only them, are neither locally symmetric nor locally isometric to a left-invariant Lorentz metric on a $3$-dimensional Lie group. We characterize homogeneous plane waves in dimension $3$, and prove they are non-extendable, and geodesically complete only if they are symmetric. Finally, only one non-flat plane wave has a compact model. 

 \end{abstract}
\maketitle
  \tableofcontents

\section{Introduction}\label{s1}
In dimension two, homogeneous pseudo-Riemannian spaces have constant curvature, and so they have the maximal possible symmetries. This is no longer true in higher dimensions, some homogeneous spaces are more homogeneous than others!  This applies to any pseudo-Riemannian type, but let us focus on the Lorentzian case in dimension three. We have two (non-disjoint) distinguished classes:
 
$\bullet$ On the one hand, the symmetric spaces. Among them, the most symmetric are those of constant sectional curvature, they are three, up to scaling, and they have isometry groups of maximal dimension equal to six. We then have three decomposable symmetric spaces, and two reducible non-decomposable (Cahen-Wallach) spaces having solvable isometry groups of dimension four.

$\bullet$ On the other hand, there is the wide world of {\it Lorentz groups}, i.e. Lie groups of dimension three endowed with a left-invariant Lorentz metric. They, generically, have an isometry group equal to the supporting group itself, but can sometimes have extra isometries.

One may ask whether any homogeneous Lorentz space of dimension $3$ fits into one of these two classes, that is, whether any homogeneous Lorentz space of dimension three is locally symmetric or locally isometric to a left-invariant metric on a Lie group of dimension three. This was proven in the Riemannian case in \cite{Sek} and in the Lorentzian case, assuming completeness, in \cite{Cal1}.

The initial motivation of the present work was, on the one hand, to provide a simpler proof of this fact in the Lorentzian case, say a Lie theoretical proof instead of using tools
from the homogeneous structures theory, which involves complicated computations. On the other hand, we aimed to fix some completeness as well as local vs. global issues in the literature. Trying this, we realized that there is a one-parameter family of examples of (incomplete) homogeneous Lorentz spaces that are neither locally symmetric nor locally isometric to Lorentz groups. In \cite{DZ} the same family was ruled out incorrectly in  the proofs of the classification of homogeneous Lorentz geometries having compact models, which we fix here without altering the statements  (see  \ref{Thurston}). 
 
\subsection{Results on plane waves} Our first results concern the global structure of $3$-dimensional \textit{plane waves}. Most studies on the subject have a mathematical-physical origin and usually consider infinitesimal and local aspects, e.g. Killing fields, $\dots$ \cite{BO, Blau, GRT}, instead of global isometry groups as in our present work. Actually, \cite{BO} contains classification results on homogeneous plane waves of arbitrary dimension, but from the point of view of theoretical physics. 
 
\subsubsection{Heisenberg extensions}

Let $\Heis$ denote the $3$-dimensional Heisenberg group and $\heis$ denote its Lie algebra. Consider a semi-direct product $G_{\rho}= \R \ltimes  \Heis$, where $\R$ acts on $\Heis$ via a representation $t \to \rho(t)\in \Aut(\Heis)$. 
\begin{theorem}\label{Theorem 1.1} Let $I$ be any non-central one parameter subgroup of $\Heis$ generated by a vector $W \in \heis$ and put $\P_{\rho}=G_{\rho}/I$. Then we have
 \begin{enumerate}
     \item The $G_{\rho}$-action on $\P_{\rho}$ preserves a Lorentz metric if and only if the projection of $W$ to $\heis/ \ZZ$ is not an eigenvector of the action of $\rho(t)$ on $\heis/ \ZZ$ where $\ZZ$ is the center of $\heis$. In addition, this Lorentz metric is unique up to scaling and automorphisms fixing $W$.
 \item There is no $3$-dimensional Lie group acting on $\P_{\rho}$ isometrically with an open orbit if and only if the action of $\rho(t)$ on $\heis/ \ZZ$ is by similarities with non-real eigenvalues.
 \item The space $\P_{\rho}$ is symmetric exactly for those representations $t \to \rho(t)\in \Aut(\Heis)$ fixing the center $\ZZ$, or equivalently, $G_{\rho}$ is unimodular. In this case
 \begin{itemize}
     \item $\P_{\rho}$ is the Minkowski space if $\rho$ is unipotent.
     \item $\P_\rho$ is a hyperbolic or an elliptic Cahen-Wallach space.
 \end{itemize}
 And besides the globally symmetric cases, $\P_{\rho}$ is locally symmetric exactly for the non-unimodular case for which $G_{\rho}$ is isomorphic to $\Aff\ltimes \R^2$, where $\Aff$ is a copy in $\O(1, 2)$ of the affine group $\Aff^+(\R)$, preserving a degenerate plane  $\R^2 \subset \R^3$.
  In this case, $\P_{\rho}$ is globally isometric to half Minkowski.

 \item  $\P_\rho$  is isometric to the Minkowski space exactly when $\rho$ is unipotent. In this case, The group $G_\rho$ is isomorphic to $N \ltimes \R^3 \subset \O(1, 2) \ltimes \R^3$ where $N$ is a unipotent one-parameter subgroup of $\O(1, 2)$.

 \item  There is only one flat non-complete case, isometric to half Minkowski, for which the group $G_\rho$ is isomorphic to $\Aff \ltimes \R^2 \subset \O(1, 2) \ltimes \R^3$.
  This corresponds to any non-unimodular representation  $t \to \rho(t)\in \Aut(\Heis)$
 having a fixed vector in $\heis$.
 
 \item Except the Minkowski case, $\Isom_0( \P_\rho) = G_\rho$.   
 \item Two spaces $\P_\rho$ and $\P_{\rho^\prime}$ are isometric if and only if $G_\rho$ and $G_{\rho^\prime}$ are isomorphic.
 
  \item Besides the two flat cases, two spaces $\P_\rho$ and $\P_{\rho^\prime}$ are locally  isometric if and only if $G_\rho$ and $G_{\rho^\prime}$ are isomorphic.
  
  \end{enumerate}

  \end{theorem}

\begin{remark} Some of the results in Theorem \ref{Theorem 1.1} are known and we added them here for completeness.  In particular, flat non-complete homogeneous connected Lorentz manifolds are classified in \cite{DI}. They are quotients of  half Minkowski, whose boundary is a lightlike hyperplane, by discrete translation subgroups.\end{remark}

  \subsubsection{Plane waves}  \label{intro_plane_waves} Plane waves are special pp-waves which in turn are special Brinkmann spacetimes.  All these classes play a central role in General Relativity both from the mathematical and applied sides, especially since the discovery of gravitational waves.  
  
A plane wave spacetime is defined by having a null parallel vector field $V$ such that there are local \textit{Brinkmann coordinates} in which the metric has the form $$2 dv du + \sum_{}^{}H_{ij}(u)x^{i}x^{j}  du^2 +  \sum_{i= 1}^{n-2} (dx^i)^2$$ where $V$ corresponds to $\frac{\partial}{\partial v}$. They can also be defined by having \textit{Rosen} coordinates in which the metric has the form $2 dv du + g_u, $ where $g_u$ is a flat metric depending on $u$ i.e. of the form $g_u = g_{ij}(u) dx^i dx^j$.  In an intrinsic  way, a plane wave is characterized by having a null parallel vector field $V$ such that its orthogonal distribution $V^\perp$ has flat leaves and, in addition, the space is ``almost symmetric'' i.e. $\nabla_W R = 0$ for any $W$ tangent to $V^\perp$, where $R$ is the Riemannian tensor.

The following classification result is due to \cite{BO}  (at the Lie algebra level). We yield here a self-contained proof which we think is more limpid, especially regarding globality questions.

  \begin{theorem} \label{plane_waves}
 Let $(M,g)$ be a simply connected homogeneous Lorentzian three-dimensional manifold. Then $(M,g)$ is a plane wave if and only if it is globally isometric to some $\P_{\rho}$. If a homogeneous plane wave $M$ is not simply connected, then it is either flat or isometric to a cyclic central quotient of one of the Cahen-Wallach spaces.
 
\end{theorem}

In fact, it is the center of the Heisenberg algebra which plays the role of the null parallel vector field. Observe, however, that the Heisenberg center is central in $G_\rho$ exactly when $\P_\rho$ is a symmetric space (see Lemma \ref{symmetric}).

\subsubsection{Literature}
There are lots of works on homogeneous plane waves. The most important is surely  \cite{BO} by M. Blau and M. O'Loughlin, in 2003. It has very interesting mathematical content although it is intended for a mathematical physics audience. A relatively recent work in a pure mathematical context is \cite{GL}, where the authors W. Globke and  T. Leistner, aimed to know when a homogeneous pp-wave is in fact a plane wave. 

We are inspired by these two references, and especially by \cite{BO} with which we have natural overlaps at many places. Our approach is however different. The authors in \cite{BO} start with a plane wave in local coordinates (of Brinkmann or Rosen type) and ask when it is (infinitesimally) homogeneous. By a direct analysis of the Killing equations, they determine the isometry algebra as generated by Killing fields satisfying some bracket relations, and prove that it contains a Heisenberg algebra.
It appears then that the Killing algebra contains an $\R$-extension of the Heisenberg algebra (although the action of the one-parameter group of (infinitesimal) automorphisms on the Heisenberg algebra is not visible). 

To exhibit simply connected homogeneous plane waves explicitly as Lorentz homogeneous spaces, which is the aim of Theorem \ref{plane_waves}, one has to determine the isometry group of such structures. This could be achieved by identifying the extended Heisenberg algebra in \cite{BO}, and studying the completeness of the Killing fields (which in general is difficult to check). 
However, the approach we adopt here to the classification result is different and (mostly) Lie theoretical, dealing directly with groups.  Our starting point is to consider a Lorentz homogeneous space of dimension $3$, assume (after reduction) that it has a $4$-dimensional isometry group, observe that (generally) this $4$-dimensional group is a Heisenberg extension,  and then bring out its plane wave structure. This approach has the advantage of providing the homogeneous space directly, overcoming the (highly non-trivial) analysis of the Killing equation on the one hand, and the completeness problem of the Killing fields on the other hand (to see the difference in a toy example, observe that the Euclidean plane and any of its open subsets share the same Killing algebra, but this open set has generically a trivial isometry group).

We hope that our present article plays a role of complement and companion of Blau and  O'Loughlin's paper in this subject of homogeneous plane waves.  Of course, contrary to \cite{BO},  we do not deal here with physical aspects like supergravity and quantization... and conversely, \cite{BO}  does not cover geometric (and global) aspects like completeness, extendibility and existence of compact quotients.

\subsubsection{Some terminology}\label{terminology} It appears worthwhile for us to give names to the plane waves classes in the following way \begin{itemize}
 \item[$\circ$] The flat complete case corresponding to $\rho$ unipotent.
 
 \item[$\circ$]  The Cahen-Wallach, hyperbolic or elliptic, case corresponding to $\rho$ unimodular semi-simple on $\heis/\ZZ$.
 
\item[$\circ$] In the remaining non-unimodular cases, let us call $\P_\rho$ a \textit{hyperbolic} (resp. \textit{elliptic}, resp.  \textit{parabolic}) non-unimodular plane wave, according to the action of $G_{\rho}$ on $\heis / \ZZ$ being diagonalizable with real eigenvalues (resp. non-real eigenvalues, resp. real non-diagonalizable).
   \end{itemize}

\subsubsection{Global coordinates system}\label{coor} It is known that a Cahen-Wallach symmetric space has a global chart where the metric has the form $g= 2 dudv\pm x^2du^2+dx^2 $ 
 (the minus sign corresponds to the elliptic case and the plus to the hyperbolic one, see for instance \cite{KO}). The global coordinate system, in the next theorem, already appears in \cite{BO}.
 
\begin{theorem}  \label{coordinates} Assume $\rho$ is non-unimodular. Then $\P_\rho$ has a global coordinate system $(u, v, x) \in \R^+ \times \R \times \R$ where the metric has the form
  $$
g= 2 dudv+\frac{b(\rho)x^2}{u^2}du^2+dx^2.
$$
    The invariant $b(\rho)$ is defined by $\rho(t) = e^{t A}$ where  $ A= \left( \begin{array}{ccc}
1 & 0 & 0 \\ 
0& 0 & 1\\ 
0 & b & 1
\end{array} \right)$ acting on the Lie algebra $\heis$ (endowed with the basis $\left\{Z,X,Y\right\}$, with $[X, Y] = Z$). Moreover, $G_{\rho}$ and $G_{\rho^{'}}$ are isomorphic if and only if $b(\rho)=b(\rho^{'})$.
  
  \end{theorem}
  
Observe that one can modify any derivation  $A $  of    $\heis$ with $A(Z) \neq 0$ by $\alpha A + \ad_u$, $u \in \heis$ and gets a matrix of the form above (the two derivations generate isomorphic semi-direct products).\begin{itemize}
      \item The case $b(\rho) = 0$ corresponds to $G_{\rho} = \Aff \ltimes \R^2$, where $\Aff \subset \SL(2, \R)$ is the affine group acting by its usual representation in $\R^2$. The space $\P_{\rho}$ in this case is isometric to the half Minkowski space $\left\{(u, v, x), u >0\right\}$ endowed with the metric $2 du dv + dx^2$.
      
      \item For $b < -\frac{1}{4}$, the space $\P_\rho$ is elliptic i.e. the $A_{\rho}$-action on $\heis/\ZZ$ has non-real eigenvalues and for $b>-1/4$ the space $\P_\rho$ is hyperbolic i.e. $A_{\rho}$ has two different real eigenvalues.
      \item For $b = -\frac{1}{4}$, $\P_\rho$ is the parabolic non-unimodular plane wave i.e. $A_{\rho}$ is conjugate to  $$\left( \begin{array}{ccc}
1 & 0 & 0 \\ 
0& 1/2 & 1\\ 
0 & 0 & 1/2
\end{array} \right).$$  \end{itemize}

  \subsubsection{Compact models of homogeneous plane waves}

  \begin{theorem} \label{no_compact}
There is exactly one non-flat plane wave that admits a compact model. This is given in Brinkmann coordinates defined on the half space $\left\{ u >0\right\}$ as
$$g= 2 dudv+\frac{2 x^2}{u^2}du^2+dx^2.$$  
In Rosen coordinates, the metric takes the form $g = 2 du dv + u^{-2} dx^2$. 

Its isometry group contains a copy of the group $\SOL = \SO_0(1, 1) \ltimes \R^2$ acting properly.  Any lattice $\Gamma \subset \SOL$ gives a compact quotient $\Gamma \backslash \P$. All compact models arise in this way.  

 \end{theorem}

 \begin{remark}  \label{remark_symmetric} Non-existence of compact quotients of $3$-dimensional non-flat Cahen-Wallach spaces is already known in \cite{KO} and \cite{DZ} (See Remarks \ref{remark completeness} and \ref{remark compact} for further discussion).\end{remark}

\subsubsection{Completeness and extendibility}

\begin{theorem} \label{Causality} The non-unimodular homogeneous plane waves are not complete. More precisely \begin{itemize}

 \item[a)] Any null geodesic, which is not an orbit of the parallel field $\frac{\partial}{\partial v}$, is incomplete. 
 \item[b)] Any timelike geodesic is incomplete. 
 \end{itemize}   
 In addition, a homogeneous plane wave $\P_{\rho}$ is non-extendable.  
 \end{theorem}

\subsection{Classification of homogeneous Lorentz 3-spaces} Coming back to  the issue in the 
classification of homogeneous Lorentz 3-manifolds, we have
\begin{theorem}  \label{classification} Any simply connected homogeneous Lorentz $3$-manifold is globally isometric to 
\begin{itemize}
     \item[$\bullet$] A Lorentz group i.e. a $3$-dimensional Lie group endowed with a left-invariant Lorentz metric, 
     \item[$\bullet$] Or, a globally symmetric space,
     \item[$\bullet$] Or, a non-unimodular elliptic plane wave i.e. $\P_{\rho}$ with $\rho$ acting on $\heis/ \ZZ$  by similarities with non-real eigenvalues. In this case, $\P_{\rho}$ is neither locally symmetric nor locally isometric to a Lorentz group.
\end{itemize}
\end{theorem}
In other words, we add elliptic non-unimodular plane waves to the existing classifications in the literature (e.g. in \cite{Cal1}). More importantly, we drop the completeness assumption usually made to get such a classification.

\subsubsection{Homogeneous locally symmetric case} It is natural to ask what exactly happens in the locally symmetric case.

\begin{theorem}  \label{locally_symmetric} Let $X$ be a Lorentz symmetric space of dimension $3$. Homogeneous (simply connected) spaces $M$ modeled on $X$ correspond to open orbits of  connected  closed subgroups $L\subset \Isom(X)$. Such $L$ may have dimension $3, 4, 5$ or $6$.

\begin{itemize}
    \item[$\bullet$] Let $L$ be a subgroup of $\Isom(X)$ of dimension $4$. Then $L$ has at least one  open orbit. It can have a discrete quantity of  $2$-dimensional orbits, which are furthermore  (complete) totally geodesic lightlike  surfaces in $X$. 
    \item[$\bullet$] If $\dim L \geq  5$ then it acts transitively and $X$ has constant sectional curvature.
\end{itemize}

\end{theorem}

\subsubsection{Locally homogeneous compact Lorentz manifolds} \label{Thurston}

 In the vein of Thurston's eight Riemannian homogeneous geometries in dimension $3$, the authors 
 of  \cite{DZ} study Lorentz geometries in dimension $3$ possessing compact forms, that is,  pairs $(G, X)$ with $X$ a Lorentzian $3$-manifold on which $G$ acts isometrically transitively and there exists some compact manifold that can be modeled on this $(G, X)$-structure. If the isotropy of $G$ acting on $X$ is compact then $G$ preserves a Riemannian metric and we find ourselves with Thurston's geometries. It is then natural to assume the Lorentz geometry to be of non-Riemannian type, that is, it has non-compact isotropy.
 \begin{theorem}[\cite{DZ}] \label{DZ} Let $(M, g)$ be a locally homogeneous compact Lorentz $3$-manifold of non-Riemannian type. Then, $M$ admits a Lorentz metric, probably different from $g$, of (non-positive) constant sectional curvature. More precisely, up to finite cover, $(M, g)$ is isometric to \begin{itemize}
     \item[$\bullet$] An anti de Sitter manifold $\Gamma\backslash\operatorname{AdS_3}$, where $\operatorname{AdS_3}$ is $\PSL(2, \R)$ endowed with its Killing form, and $\Gamma \subset \PSL(2, \R) \times \PSL(2, \R)$ acts properly co-compactly on it. 
 \item[$\bullet$] Or $M$ is a quotient $\Gamma\backslash L$ where $L $ a 3-dimensional Lie group and $\Gamma \subset L$ is a co-compact lattice,  and metrically $L$ is endowed with a left-invariant Lorentz metric which fits into one of the following cases
 \begin{enumerate} 
 \item $L= \R^3$ endowed with a flat metric.
 \item $L= \Heis$ or $L = \SOL$. In both cases, $L$ has (up to isomorphism) two left-invariant  metrics with non-compact isotropy, one of which is flat.
 \item $L= \PSL(2, \R)$. It has, up to isomorphism, two left-invariant metrics of non-compact isotropy, in addition to the $\operatorname{AdS}$-one which is given by its Killing form.
\end{enumerate}
  \end{itemize}
\end{theorem} 
In the proofs of \cite{DZ}, some cases of Heisenberg extensions (actually corresponding to elliptic plane waves) were ruled out incorrectly because they were treated as algebraic groups (while they are not)! The present proof (Theorem \ref{no_compact}) of  the non-existence of compact models of elliptic plane waves fills this gap.

Let us finally observe that the algebraically most difficult case is the anti de Sitter one. It is indeed delicate to describe discrete subgroups  $\Gamma \subset \PSL(2, \R) \times \PSL(2, \R)$ acting properly co-compactly on $\PSL(2, \R)$ (see \cite{DGK, Sch} for recent results).

\subsection{Organization of the  article} We start Sect. \ref{Riem} by classifying homogeneous $3$-dimensional Riemannian manifolds, essentially for two purposes: to provide a simpler proof (compared to the ones in the literature, for instance  in {\cite{Sek}), and to compare with the Lorentzian situation to see what is new.

In Sect. \ref{Heis} we define our family of Lorentz homogeneous spaces $\P_{\rho}$, as quotients of Heisenberg extensions. These spaces will play the major role in the two main parts of this paper, on the one hand, they are the spaces of discussion in the part about plane waves and, on the other hand, a sub-family of them will constitute the new examples that are neither locally symmetric nor locally isometric to Lorentz groups.

Sections  \ref{solv} and  \ref{non-solv} will be devoted to the proof of the classification theorem  \ref{classification}. For easiness, the proof is divided into two parts:  Sect. \ref{solv} treats the case where the isometry group is solvable and Sect. \ref{non-solv} treats the non-solvable case.

Proof of Theorem \ref{locally_symmetric}, about the classification of homogeneous locally symmetric Lorentz $3$-manifolds, is presented in Sect. \ref{proof-loc-sym}.

Theorem \ref{Theorem 1.1} concerning the algebraic structure of the spaces $\P_{\rho}$ is proven in Sect. \ref{prof 1.1}.

Proofs of the statements about plane waves start, essentially, from  Sect. \ref{planewaves 8}. Sect. \ref{planewaves 8} is devoted to the proof of Theorem \ref{plane_waves} that characterizes all homogeneous $3$-dimensional plane waves. Theorem \ref{coordinates}, about the global coordinates for non-unimodular plane-waves, is proven in Sect. \ref{proof coor}. Proof of Theorem \ref{Causality} is divided into two parts: the completeness statement is proven in Sect. \ref{s10}, and the non-extendibility statement is discusses in  Sect. \ref{s11}. Finally, proof of the statement of the (non-)existence of compact models (Theorem \ref{no_compact}) is detailed in Sect. \ref{s12}.
 
 \subsection*{Acknowledgments} We thank Ines Kath for answering our question on which spaces $\P_\rho$ are Cahen-Wallach, which was a crucial step in the development of our work. The first author is supported by the SFB/TRR 191 \textit{Symplectic Structures in Geometry, Algebra and Dynamics}, funded by the DFG (Projektnummer 281071066 - TRR 191).

 \section{What's new in the non-Riemannian case.}\label{Riem}
As we mentioned before, it is known that any homogeneous simply connected Riemannian manifold of dimension three is either symmetric or isometric to a left-invariant metric on a Lie group of dimension three, see \cite{Sek}. Comparing this to the Lorentz situation led us to new examples. This section is devoted to this comparison. 
 
\subsection{The Riemannian case} We give here a proof of the classification of three-dimensional homogeneous simply connected Riemannian manifolds.
 \begin{theorem}
 A simply connected homogeneous Riemannian 3-manifold $(M,g)$ is globally symmetric  or isometric to a Lie group with a left-invariant Riemannian metric.
 \end{theorem}
 
 \begin{proof} Denote by $G$ the identity component of the full isometry group of $(M,g)$ and $\g$ its Lie algebra. Let $m$ be a point in $M$ whose stabilizer in $G$ is denoted by $I$ (which is a connected compact subgroup of $G$). The isotropy representation of $I$ in $\SO(g_m)\simeq \SO(3)$ is faithful. Since $\SO(3)$ has no two-dimensional subgroups then $I$ is either three-dimensional, one-dimensional, or trivial. If the dimension is zero then we are done (i.e. $(M,g)$ is a Riemannian Lie group). If $I$ is three-dimensional then the sectional curvature is constant which means that $M$ is globally isometric to one of the Riemannian models with constant sectional curvature.

 Suppose now that $I$ is one dimensional i.e. $G$ is four-dimensional.
 \begin{fact}
  The solvable radical of $\g$, denoted by $\operatorname{R}(\g)$, is non-trivial.
 \end{fact}
 \begin{proof}
 If not, then $G$ is by definition a semi-simple Lie group of dimension $4$ which is impossible. 
 \end{proof}
 
If $\dim(\operatorname{R}(\g))=1$ then by Levi's decomposition either 
$\g=  \mathfrak{su}(2) \ltimes \R$ or $\g=   \mathfrak{sl}(2,\mathbb{R}) \ltimes \R$. But, semi-simplicity of $\mathfrak{sl}(2,\mathbb{R})$ and $\mathfrak{su}(2)$ implies that the above semi-direct product is a direct product. In both cases if $\i=\Lie(I)\subset \mathfrak{su}(2)$ or $\Lie(I)\subset \mathfrak{sl}(2,\mathbb{R})$ then $I=\SO(2)$ up to conjugacy, and the tangent space of $G/I$ at $I$ is $\R  \oplus (\s/\i)$ with $\s$ equals $\mathfrak{sl}(2,\mathbb{R})$ or $\mathfrak{su}(2)$. Hence $\R$ is exactly the factor on which $I$ acts trivially and $\s/\i$ is orthogonal to $\R$ since this is the unique $I$-invariant subspace transverse to $\i$. This implies that the metric on $\R  \oplus (\s/\i)$ is a product metric. Therefore, up to scaling, $(M,g)$ is isometric to the symmetric spaces  $$\mathbb{R} \times (\SL(2,\mathbb{R})/\SO(2))=\R\times\mathbb{H}^2 \ \ \text{or} \ \ \mathbb{R} \times (\SU(2)/\SO(2))=\R \times\mathbb{S}^2.$$ And, if $\Lie(I)$ is transverse to the semi-simple factor, then $(M,g)$ is isometric to a left-invariant metric on $\widetilde{\SL}(2,\R)$ or $\SU(2)$.

If $\dim(\operatorname{R}(\g))\geq 2$ then $\g$ is solvable and $\i$ cannot be contained in $[\g,\g]$ because, if not, we would have that $\ad_h$ is nilpotent for every $h \in \i$ and it cannot be the infinitesimal isometry of a Riemannian inner product (unless $\ad_h$ is identically zero, but this is not the case since the derivative representation of $I$ is faithful). Using this along with the fact that $\g/[\g,\g]$ is abelian, one can find an ideal $L$ supplementary to $\i$ by pulling-back a hyperplane in $\g/[\g,\g]$ transverse to the projection of $\i$. Since the normal subgroup corresponding to $L$ is transverse to $I$, it follows that its action on $M=G/I$ is transitive and locally free. But, since $M$ is simply connected then the action is in fact free and $(M,g)$ is globally isometric to a left-invariant metric on this Lie group.

 \end{proof}
\subsection{The Lorentzian case}
 One may observe that many ideas in the Riemannian case can be generalized immediately to the case of homogeneous Lorentzian three-manifolds except, a priori, two of them:
 \begin{enumerate}
     \item The one concerning the fact that $I$ cannot be two-dimensional. Here, one observes that $\SO(1,2)$ contains two-dimensional subgroups (the affine group of $\R$ which is unique, up to conjugacy, inside $\SO(1,2)$). However, one still manages to show that it cannot represent the full isotropy group of a point in a Lorentzian $3$-manifold.
     \item The fact that non-zero nilpotent endomorphisms cannot be the infinitesimal isometry of a Riemannian inner product, which is possible in the Lorentzian setting. This is the main difference which leads to new examples. We will investigate all of this in what follows.
 
 \end{enumerate}
Let us first give some general considerations. Suppose that  $(M, g)$  is a locally homogeneous Lorentz three-manifold, meaning that any $p, q \in M$ have neighborhoods $U$ and $V$ such that there exists $f: U \to V$ isometric with $f(p) = q$.  In particular, such $(M,g)$ is analytic and it is known that if it is simply connected, then any locally defined Killing field extends (uniquely) globally  \cite{Nom}. In this case, let $\g$ be the Lie algebra of (all) Killing fields and $G$ its associated simply connected group. Choose a base point $p_0$ and let $\i$ be its (full) isotropy subalgebra, i.e. those Killing fields vanishing at $p_0$. We have the following
 \begin{proposition}\label{isotropy}
     The isotropy subalgebra $\i$ cannot have dimension two. 
 \end{proposition}
 
 \begin{proof}
The idea relies on the fact that if we have a representation of $\SL(2,\R)$ on some finite dimensional vector space $V$ such that there is $v\in V$ fixed by a two-dimensional (connected) subgroup $H\subset \SL(2,\R)$ (which is unique up to conjugacy) then $v$ is fixed by $\SL(2,\R)$ (this is due to the fact that $\SL(2,\R)/H \sim \mathbb{S}^1$ and $\SL(2,\R)$ cannot have non-trivial compact orbits when acting linearly on $V$). We have in our situation that the faithful representation, of local isometries fixing $p$, contains a $2$-dimensional subgroup inside $\SO_0(1,2)$ that fixes the Ricci tensor  $\operatorname{Ric}_p$. One deduces from the previous discussion that $\operatorname{Ric}_p$ is fixed by all $\SO_0(1,2)$. But any quadratic form invariant by $\SO_0(1,2)$ is proportional to the metric itself. So our Lorentz manifold is Einstein and it is known that, in dimension three, this is equivalent to having constant sectional curvature which implies that $\i$ is not the full isotropy subalgebra.  
 \end{proof}
 
\label{closed} So, if $(M, g)$ does not have constant sectional curvature then $\dim (I)$ can be $0$ or $1$ and, hence $\dim (G) = 3$ or $4$. It is known that in a simply connected Lie group $G$ of dimension $\leq 5$, any connected Lie subgroup is closed. Indeed, if $G$ is solvable then this is true in every dimension (see \cite{Mal} page 187, last paragraph). If $\dim(G)=4$ and $G$ is not solvable, then its Lie algebra is either $\mathfrak{su}(2)\oplus \R$ or $\mathfrak{sl}(2,\R)\oplus \R$ and in both cases we know that any connected Lie subgroup is closed. If $\dim(G)=5$ and $G$ is not solvable, then its Lie algebra is either $\mathfrak{su}(2)\oplus \R^2$, or $\mathfrak{sl}(2,\R)\oplus \R^2$, or $\mathfrak{sl}(2,\R)\ltimes \R^2$ and, again, the same holds in this case. But, observe that this is no longer true in higher dimensions as can be shown in dimension $6$, with  $I$ a dense subgroup in $\SO(2) \times \SO(2) \subset \SU(2) \times \SU(2)$, which is simply connected. Therefore, in our situation, the quotient $G/I$ exists and $M$ is locally isometric to it.\\\\
Observe that one can also find simply connected Lorentz homogeneous spaces $M= G/I$ with $\dim (G) = 5$ and $\dim (I) = 2$. But in this case, the Lie algebra $\g$ of $G$ is not the full Killing algebra of $M$. The full Killing algebra $\g^\prime$  has dimension $6$ and $M$ has constant curvature, but its associated group $G^\prime$, a priori,  acts only locally on $M$. This is because $M$ could be incomplete. However,  one can prove that a constant curvature manifold admitting an isometric action of  a $5$-dimensional group is complete (see Theorem \ref{locally_symmetric}).
\section{The Heisenberg extensions.}\label{Heis}
\begin{definition} We say that a  semi-direct product $G= \R \ltimes \Heis$ is non-real  if the $\R$-action on $\R^2 = \Heis/ \ZZ$  (where $ \ZZ$ denotes the center of $\Heis$) has complex non-real spectrum.  Equivalently, this action is conjugate to a non-trivial similarity action, i.e. a one-parameter group $R \subset \R^{+} \times \SO(2) \subset \GL(2, \R)$ with a non-trivial projection on $\SO(2)$.  For example, if $R$ is $\SO(2)$, then $G$ is the oscillator group \cite{Str, BM}. In the opposite case, we say that $G$ is real.
\end{definition}

 \begin{proposition} \label{semi_direct_Heisenberg}
 
 Let $I$ be any non-central one parameter group of $\Heis$ with $\Lie(I)=\i$. Let $G$ be any semi-direct product $G = \R \ltimes \Heis$. Then
\begin{enumerate}
    \item $G/I$ admits a $G$-invariant Lorentz metric if and only if the $\R$-action does not preserve the abelian group of rank $2$ generated by $I$ and $\ZZ$.
    \item Assume $G/I$ has a $G$-invariant Lorentz metric. Then there is a subalgebra $\l$ of dimension $3$ transverse to $\i$ if and only if $G$ is real.
\end{enumerate}

 \end{proposition}

 \begin{lemma}\label{lem}
 
 A nilpotent (non-zero) endomorphism $A$ of a $3$-dimensional linear space is an infinitesimal isometry of some Lorentz scalar product if and only if its nilpotency order equals $3$, that is, $A^3 = 0$ but $A^2 \neq 0$.
\end{lemma}

\begin{proof}

 If $A$ has nilpotency order equals $3$ then, for $u$  generic, the system $\left\{ A^2(u), A(u), u  \right\}$
 is a basis in which $A$ has a matrix  
$ \left( \begin{array}{ccc}
0 & 1 & 0 \\ 
0& 0 & 1\\ 
0 & 0 & 0
\end{array} \right)$. This is skew symmetric with respect to the Lorentz form $-2dxdz+dy^2$.  Suppose now that $A$ has nilpotency order two and $g$ is an invariant Lorentz metric. Then $\operatorname{Im}(A)$ is one-dimensional. But, for arbitrary $X$ we have $$g(A(X),X)+g(X,A(X))=2g(A(X),X)=0.$$ 
So every $X$, not in $\ker(A)$, is orthogonal to $\operatorname{Im}(A)$, which is impossible since $g$ is non-degenerate.

  \end{proof}

\textit{Proof of} (1). $G$ preserves a Lorentz metric on $G/I$ if and only if the $\Ad(I)$ action on the quotient vector space $\g/ \i$  (identified to the tangent space at the base point $I$) preserves a Lorentz scalar product. Equivalently, the infinitesimal action $\ad_a$ is skew symmetric with respect to a Lorentz scalar product on  $\g / \i$, where $a$ generates $\i$. Consider a standard generating system $\left\{Z, X, Y\right\}$ of the Heisenberg algebra $\heis$ with bracket relation $Z = [X, Y]$. The Lie algebra $\g$ is generated by adding a fourth element $T$ (with brackets $[T,  W] \in \heis$, for any $W \in \heis$). Since all one-dimensional subspaces in $\Heis$ different from $\R Z$ are equivalent by automorphisms, we can assume $\i = \R X$. The action of $\ad_X$ on $\g/ \R X$ has nilpotency degree $3$ if and only if 
 $\ad_X(T)  \notin \R X \oplus \R Z$.   This proves the first part of the proposition.

\textit{Proof of} (2). If $\ad_T$ acting on $\heis / \R Z \cong \R X \oplus \R Y$ has a non-vanishing eigenvector $u$ then $u \notin \i \oplus \R Z$, since 
 we assume the existence of an $\ad_\i$-invariant scalar product. Thus, $\l = \R T \oplus \R u \oplus \R Z$ is a Lie algebra transverse to $\i$. Conversely, suppose that such $\l$ exists and consider its intersection $\l^\prime = \l \cap \heis$. Then $\l^\prime$ is a $2$-dimensional ideal in $\l$. Thus, $\l^\prime$ contains $Z$ and another element $u \in \R X \oplus \R Y$. It also contains and element $T^\prime = T + v$, $v \in \heis$. Now,
 $ [T, u] + [v, u]=[T^\prime, u]  \in \l^{'}=\R Z\oplus\R u $. Since $[v, u] \in \R Z$ then the class of $u$ is an eigenvector of $\ad_T$ acting on $\heis/ \R Z$. Therefore, this last endomorphism has a real spectrum. $\Box$

\subsection{Description of the Heisenberg extensions}
Let $\rho: \R \to \Aut(\heis)$ be a representation into the automorphism group of the Heisenberg algebra. This one-parameter group is of the form $\exp(tA_{\rho})$ where $A_{\rho}$ is a derivation of the Heisenberg algebra. Put $G_{\rho}=\Heis \rtimes_{\rho} \R$ the associated Lie group with Lie algebra, denoted by $\Lie(G_{\rho})$, generated by $Z, X,Y$ and $T$ such that $$[X,Y]=Z \ \ \text{and} \ \ [T,W]=A_{\rho}(W) \ \ \text{for every} \ \ W\in \heis. $$
  In particular, $[\Lie(G_{\rho}),\Lie(G_{\rho})]\subset \heis$ (more precisely, the derived subalgebra is generated by the span of the image of $A_{\rho}$ and $Z$ ). Let $V\in \heis$ be a non-central element. Hence $\ad_V:\Lie(G_{\rho})\to \Lie(G_{\rho})$ is nilpotent ($\heis$ is contained in the nil-radical and we have equality if $A_{\rho}$ is not nilpotent). Suppose that $\ad_V$ has nilpotency degree $3$ and let $I\subset G_{\rho}$ be the one-parameter subgroup (necessarily closed) tangent to $V$ and consider the homogeneous space $\P_{\rho}=G_{\rho}/I$. It follows from Proposition \ref{semi_direct_Heisenberg} that $\P_{\rho}$ possesses a $G_{\rho}$-invariant Lorentz metric. In what follows, we will show the uniqueness of $\P_{\rho}$.

  Observe that we always have $A_{\rho}(Z)=\alpha Z$ and $A_{\rho}$ induces a linear transformation $$\widetilde{A_{\rho}}: \heis/Z=\R^2\to \R^2$$
Where $\tr(\widetilde{A_{\rho}})=\alpha$  .
\subsubsection{The real diagonalizable case}\label{hyperbolic case} Suppose that $\widetilde{A_{\rho}}$ is diagonalizable with real eigenvalues. Then, up to conjugacy by automorphisms, we have $$A_{\rho}=\begin{pmatrix}
a+b &0 &0 \\ 
0 &a &0 \\
0 &0 &b
\end{pmatrix}$$ with respect to the basis $\left\{Z,X,Y\right\}$ where $a\neq b$ (since we are in the situation of the existence of elements in $\heis$ with nilpotency order $3$). We can assume $a\neq0$ (observe that in this case $G_{\rho}$ is unimodular if and only if $a=-b$). Furthermore, we can assume that $a=1$ because scaling $A_{\rho}$ doesn't change the one-parameter group. Thus, we have a one-parameter family of Lie algebras, denoted $\Lie(G_b)$ this time, given by the derivations $$A_{b}=\begin{pmatrix}
1+b &0 &0 \\ 
0 &1 &0 \\
0 &0 &b
\end{pmatrix}$$ for $b$ different from $1$ (observe that $b$ and $1/b$ give isomorphic algebras for $b\neq 0$). The brackets of this Lie algebra are given by $$[T,Z]=(1+b)Z, \ \ [T,X]=X, \ \ [T,Y]=bY, \ \ \text{and} \ \ [X,Y]=Z$$
Let now $V\in \heis\subset \Lie(G_b)$ be a non-central element such that $\ad_V$ has nilpotency order three. Put $V=\alpha X+\beta Y+\gamma Z$. We can assume, up to inner conjugacy in $\heis$, that $\gamma=0$ and observe that both $\alpha$ and $\beta$ must be different from zero otherwise $\ad_V$ has order two. In addition, all such $V$ are equivalent by automorphisms of $\Lie(G_b)$, it suffices to consider an automorphism $\varphi:\Lie(G_b)\to\Lie(G_b)$ given by $\varphi(X)=t_1X$, $\varphi(Y)=t_2Y$, $\varphi(Z)=t_1t_2Z$, and $\varphi(T)=T$.\\\\
Suppose now that $V=X+Y$ and put $m=\Span(T, Y^{\prime}=X+bY, Z)$. Then $m$ is $\ad_V$-invariant and we have $$\ad_V(T)=-Y^{\prime}, \ \ \ad_V(Y^{\prime})=(b-1)Z, \ \ \text{and} \ \ \ad_V(Z)=0$$
So giving a $G_b$-invariant Lorentzian metric on $P_b=G_b/I$, where $I=\exp(tV)$, is equivalent to giving a Lorentzian $\ad_V$-invariant inner product $g(.,.)$ on $m$. That is $$g(\ad_V(W),U)+g(W, \ad_V(U))=0$$ for every $W,U\in m$. By direct evaluations we find $$g(T,T)=\beta, \ \ g(Y^{\prime},Y^{\prime})=\alpha, \ \ g(T,Z)=\frac{\alpha}{b-1}$$ and $$g(Z,Z)=g(Y^{\prime},Z)=g(T,Y^{\prime})=0$$
where $\alpha$ is different from zero. But, up to automorphisms fixing $V$ and preserving $m$, we can assume $\beta=0$. It suffices to consider the automorphism $\varphi:\Lie(G_b)\to \Lie(G_b) $
defined by $\varphi(X)=X$, $\varphi(Y)=Y$, $\varphi(T)=T+\delta Z$, and $\varphi(Z)=Z$ with a suitable choice of $\delta$. This shows that the Lorentzian homogeneous space $P_b$ is unique up to isometry and scaling.
\subsubsection{The real parabolic case}\label{parabolic case} Suppose now that $\widetilde{A_{\rho}}$ has real non-zero eigenvalues but not diagonalizable. In this case, up to conjugacy by automorphisms, we have $$A_{\rho}=\begin{pmatrix}
2t &0 &0 \\ 
0 &t &s \\
0 &0 &t
\end{pmatrix}$$ with respect to the basis $\left\{Z,X,Y\right\}$ where $t$ and $s$ are non-zero. We can assume that $t=s=1$ up to scaling and automorphisms of the Heisenberg algebra. Thus, we have a Lie algebra $\Lie(G_{\rho})$  given by the derivation $$A_{\rho}=\begin{pmatrix}
2 &0 &0 \\ 
0 &1 &1 \\
0 &0 &1
\end{pmatrix}$$
The brackets of this Lie algebra are given by $$[T,Z]=2Z, \ \ [T,X]=X, \ \ [T,Y]=X+Y, \ \ \text{and} \ \ [X,Y]=Z$$
Let $V=\alpha X+\beta Y+\gamma Z$ be a non-central element whose nilpotency order is $3$. We can assume, up to inner conjugacy in $\heis$, that $\gamma=0$. Observe that $\beta\neq 0$ otherwise $\ad_V$ has order two. In addition, all such $V$ are equivalent by automorphisms of $\Lie(G_{\rho})$, it suffices to consider $\varphi:\Lie(G_{\rho})\to\Lie(G_{\rho})$ given by $\varphi(X)=X$, $\varphi(Y)=Y+tX$, $\varphi(Z)=Z$, and $\varphi(T)=T$ with a suitable value of $t$.\\\\
Suppose now that $V=Y$ and put $m=\Span(T, Y^{\prime}=X+Y, Z)$. Then $m$ is $\ad_V$-invariant and we have $$\ad_V(T)=-Y^{\prime}, \ \ \ad_V(Y^{\prime})=-Z, \ \ \text{and} \ \ \ad_V(Z)=0$$ Suppose $g(.,.)$ is a Lorentzian $\ad_V$-invariant inner product on $m$. That is $$g(\ad_V(W),U)+g(W, \ad_V(U))=0$$ for every $W,U\in m$. By evaluations we find $$g(T,T)=\beta, \ \ g(Y^{\prime},Y^{\prime})=\alpha, \ \ g(T,Z)=-\alpha$$ and $$g(Z,Z)=g(Y^{\prime},Z)=g(T,Y^{\prime})=0$$
where $\alpha$ is different from zero. We can also assume $\beta=0$ by considering the automorphism $\varphi:\Lie(G_{\rho})\to \Lie(G_{\rho}) $
defined by $\varphi(X)=X$, $\varphi(Y)=Y$, $\varphi(T)=T+\delta Z$, and $\varphi(Z)=Z$ with a suitable choice of $\delta$. This shows, similar to the previous cases, that the Lorentzian homogeneous space $\P_{\rho}=G_{\rho}/\exp(tV)$ is unique up to isometry and scaling.
\subsubsection{The non-real case}\label{elliptic case} Consider the case where $\widetilde{A_{\rho}}$ has complex eigenvalues. In this case $A_{\rho}$ preserves a plane transverse to the central direction and acting on it by a similarity. Since all transverse planes are equivalent by Heisenberg inner automorphisms, we can assume that, up to scaling, $$A_{\rho}=\begin{pmatrix}
2\cos(\theta) &0 &0 \\ 
0 &\cos(\theta) &-\sin(\theta) \\
0 &\sin(\theta) &\cos(\theta)
\end{pmatrix}$$ with respect to the basis $\left\{Z,X,Y\right\}$ where $\theta\in \left(0,\pi\right)$. Here, one thinks of $A_{\rho}$ acting on the plane $\Span(X,Y)$ as a multiplication by a complex number in the upper half of the unit circle. But, in order to simplify the brackets, we can scale and assume that $A_{\rho}$ acts on $\Span(X,Y)$ by multiplying by a complex number of the form $i+c$ with $c\in \R$. Thus, we have a one-parameter family of Lie algebras, which we denote by $\Lie(G_{c})$, given by the derivation (observe also that one can take only $c\geq 0$) $$A_{c}=\begin{pmatrix}
2c &0 &0 \\ 
0 &c &-1 \\
0 &1 &c
\end{pmatrix}$$
The brackets of this Lie algebra are given by  $$[T,Z]=2cZ, \ \ [T,X]=Y+cX, \ \ [T,Y]=-X+cY, \ \ \text{and} \ \ [X,Y]=Z$$
As before, if $V=\alpha X+\beta Y+\gamma Z$ is a non-central element then we can always assume that $\gamma=0$. All such $V$ are equivalent to $V=X$ by automorphisms $\varphi:\Lie(G_c)\to\Lie(G_c)$ of the form $\varphi(X)=\alpha X+\beta Y$, $\varphi(Y)=-\beta X+\alpha Y$, $\varphi(Z)=(\alpha^2+\beta^2)Z$, and $\varphi(T)=T$ which send $X$ to $V$.\\\\
Suppose now that $V=X$ and $m=\Span(T, Y^{\prime}=Y+cX, Z)$. Then $m$ is $\ad_V$-invariant and we have $$\ad_V(T)=-Y^{\prime}, \ \ \ad_V(Y^{\prime})=Z, \ \ \text{and} \ \ \ad_V(Z)=0$$ Suppose $g(.,.)$ is a Lorentzian $\ad_V$-invariant inner product on $m$. By evaluations, as before, we find $$g(T,T)=\beta, \ \ g(Y^{\prime},Y^{\prime})=\alpha, \ \ g(T,Z)=\alpha$$ and $$g(Z,Z)=g(Y^{\prime},Z)=g(T,Y^{\prime})=0$$
where $\alpha$ is different from zero. By considering an automorphism $\varphi:\Lie(G_b)\to \Lie(G_b) $
of the form $\varphi(X)=X$, $\varphi(Y)=Y$, $\varphi(T)=T+\delta Z$, and $\varphi(Z)=Z$ we can assume $\beta=0$. This shows the uniqueness of the Lorentzian homogeneous space $P_c=G_{c}/\exp(tV)$ up to isometry and scaling.

 \subsubsection{The nilpotent case}\label{nilpotent case} Lastly, suppose that $\widetilde{A_{\rho}}$ has only zero eigenvalues but not diagonalizable. We have in this case that, up to scaling, conjugacy by automorphisms, and adding a derivation of the form $\ad_u$ with $u\in \heis$ $$A_{\rho}=\begin{pmatrix}
0 &0 &0 \\ 
0 &0 &1 \\
0 &0 &0
\end{pmatrix}$$ with respect to a basis $\left\{Z,X,Y\right\}$. 
The brackets of this Lie algebra are given by $$[T,Y]=X, \ \ \text{and} \ \ [X,Y]=Z$$
Let $V=\alpha X+\beta Y+\gamma Z$ be a non-central element whose nilpotency order is $3$. We can always assume, up to inner conjugacy in $\heis$, that $\gamma=0$. Observe that $\beta\neq 0$ otherwise $\ad_V$ has order two. In addition, all such $V$ are equivalent by automorphisms of $\Lie(G_{\rho})$, it suffices to consider $\varphi:\Lie(G_{\rho})\to\Lie(G_{\rho})$ given by $\varphi(X)=\beta X$, $\varphi(Y)=\beta Y+\alpha X$, $\varphi(Z)=\beta^2 Z$, and $\varphi(T)=T$ which sends $Y$ to $V$.\\\\
Suppose now that $V=Y$ and put $m=\Span(T, X, Z)$. Then $m$ is $\ad_V$-invariant and we have $$\ad_V(T)=-X, \ \ \ad_V(X)=-Z, \ \ \text{and} \ \ \ad_V(Z)=0$$ Suppose $g(.,.)$ is a Lorentzian $\ad_V$-invariant inner product on $m$. That is $$g(\ad_V(W),U)+g(W, \ad_V(U))=0$$ for every $W,U\in m$. By evaluations we find $$g(T,T)=\beta, \ \ g(X,X)=\alpha, \ \ g(T,Z)=-\alpha$$ and $$g(Z,Z)=g(X,Z)=g(T,X)=0$$
where $\alpha$ is different from zero. We can also assume $\beta=0$ by considering the automorphism $\varphi:\Lie(G_{\rho})\to \Lie(G_{\rho}) $
defined by $\varphi(X)=X$, $\varphi(Y)=Y$, $\varphi(T)=T+\delta Z$, and $\varphi(Z)=Z$ with a suitable choice of $\delta$. This shows that the Lorentzian homogeneous space $\P_{\rho}=G_{\rho}/\exp(tV)$ is unique up to isometry and scaling.

 \section{Proof of Theorem  \ref{classification}  in the solvable case.}\label{solv}

Recall that we have shown in Proposition \ref{isotropy} that if $(M,g)$ is a homogeneous simply connected Lorentzian manifold with $\dim(\Isom(M))\geq 5$, then $M$ has constant sectional curvature. In fact, one shows in this case that $M$ is complete (see Theorem \ref{locally_symmetric}). Now , since the case $\dim(\Isom(M))=3$ corresponds to the fact that $M$ is a left-invariant metric on $\Isom_0(M)$, then all what remains  to consider is the case $\dim(\Isom(M))=4$.

 \begin{proposition} \label{transverse_subgroup} Let $G/I$  be a Lorentz space with 
 $G$ solvable and has  dimension $4$. Then there exists a Lie subalgebra $\l$ transverse to $\i$ 
 if and only if $G$ is not a non-real semi-direct product $\R \ltimes \Heis$ with $I\subset\Heis$.

 \end{proposition}
 \begin{proof}
 Observe first that $\i$ is not central since we always assume that the action of $G$ on $G/I$ is faithful. Let $[\g,\g]$ be the derived subalgebra. Since $\g$ is solvable we have that $[\g,\g]$ is a proper subalgebra of $\g$, and we have the natural projection $\pi: \g \to \g/[\g,\g]=\R^k$ for $k\geq 1$. Suppose that $\i$ is not contained in $[\g,\g]$, then $\i$ projects injectively to a line $\pi(\i)$ inside $\R^k$. Let $H\subset \R^k$ be any hyperplane transverse to $\pi(\i)$ (if $k=1$ then $H$ is reduced to $\left\{0\right\}$). Hence, $\pi^{-1}(H)$ is an ideal transverse to $\i$ in $\g$ which proves the proposition in this case.\\\\
 So, it remains to consider the case where $\i\subset [\g,\g]$. For $a\in \i-\left\{0\right\}$ we have that $\ad_a$ is nilpotent (recall that $\ad_a\neq 0$ since $\i$ is not central). If $[\g,\g]$ is abelian then $\ad_a^2 = 0$ which implies that $\ad_a$ can not infinitesimally preserve a Lorentz product (by Lemma \ref{lem}). So $[\g,\g]$ is non-abelian. Being nilpotent and non-abelian, $[\g,\g]$ is therefore the Heisenberg algebra and the group $G$
 is then a semi-direct product $\R \ltimes \operatorname{Heis}$ with $I\subset G$ is a non-central one-parameter subgroup of $\Heis$, which is already treated in Proposition \ref{semi_direct_Heisenberg}. This completes the proof of Proposition \ref{transverse_subgroup}. 
  \end{proof}
  \begin{remark}\label{ideal_supplementary} Observe from the proof that in all cases where $G$ is not a semi-direct product of $\Heis$, we found 
  $\l $ transverse to $\i$ which is an ideal.
  
  \end{remark}

 \subsection{Global identification}\label{global product}

 So far, we proved the existence of a sub-algebra $\l$ transverse to $\i$ if $G$ is not a non-real semi-direct product of $\Heis$. But, we want to prove that globally 
 $G/I$ is identified to $L$, where $L$ is the group determined by $\l$.  This essentially means that $L$ acts 
 transitively on $G/I$, whereas the transversality means that the $L$-orbit of the point $I$ in $G/I$ is open. It is likely, in our setting here (e.g. $G$ solvable and
 $G/I$ of Lorentz type), that any open orbit necessarily equals the full space. We will prove this for the $L$ we found in the proofs of Propositions \ref{transverse_subgroup} and \ref{semi_direct_Heisenberg}. 
 
 First, let us notice that for simply connected groups of dimension $4$, any connected subgroup is closed (see the discussion after the proof of Proposition \ref{isotropy}).  Let $L$ be the subgroup determined by the subalgebra $\l$.  If $L$ is normal, then $G/L$ is isomorphic to $\R$ as a group, and $I$ projects to it isomorphically. That is, $G/L = I$ or equivalently $G= L I$, and $I \cap L = \left\{1\right\}$. In other words, $L$ acts freely transitively on $G/I$.
 
 From Remark \ref{ideal_supplementary}, we can find $L$ normal unless $G$ is a semi-direct product 
 $\R \ltimes \Heis$.  In this case, let $I$ be the one-parameter group $\left\{\exp (t X), t \in \R\right\} \subset \Heis$ and the group $L$ transverse to $I$ as constructed in the proof of Proposition \ref{semi_direct_Heisenberg}. The orbit of a point $gI\in G/I$ under the action of $L$ is open if and only if $L$ is transverse to $\text{Stab}(gI)=gIg^{-1}$ where $\text{Stab}(gI)$ is the stabilizer of $gI$. We have, as in the proof of Proposition \ref{semi_direct_Heisenberg}, that $\l = \R T \oplus \R u \oplus \R Z$ and $\R u \oplus \R Z$ is an ideal in $\g$. So the conjugates of $X$ are never inside $\l$ since they are never inside the ideal $\R u \oplus \R Z=\l \cap \heis$. This shows that $L$ is always transverse to the stabilizer of any point in $G/I$ which implies that the $L$-orbit of every point $gI\in G/I$ is open. Hence $L$ acts transitively locally freely on $G/I$ because $G/I$ is connected and the orbits are open. By simple-connectedness, we get that the $L$-action is in fact free. 
$\Box$
\begin{theorem} \label{transverse_subgroup_global}  Let  $G$ be a  simply connected solvable group of  dimension $4$ and $G/I$ a Lorentz  $3$-dimensional $G$ homogeneous space. Then there exists a simply connected $3$-dimensional group $L$ acting isometrically, freely, and transitively on $G/I$ if and only if $G$ is not a non-real semi-direct product $\R \ltimes \Heis$ with $I\subset\Heis$.

 \end{theorem}

 \section{Proof of Theorem \ref{classification} in the non-solvable case.}\label{non-solv}

  Let $G/I$ be a $3$-dimensional homogeneous Lorentz space with $G$ of dimension $4$, non-solvable, and simply connected. 
  Necessarily, for the same reason that we mentioned in the Riemannian case, $G$ is isomorphic to $\R \times S$, with $S = \widetilde{\SL}(2, \R)$ or $\SU(2)$. The isotropy group $I$ cannot be 
   the factor $\R$ since it is central (and the $G$-action would be non-faithful). If $I$ is not contained in $S$, then $S$ acts freely transitively and isometrically on $G/I$ since $S$ is normal and transverse to $I$. In other words, the Lorentz space $G/I$ is identified with a left-invariant metric on $S$.\\\\
  It remains to consider the case $I \subset S$ i.e. $G/I = \R \times (S / I)$.  In this case, $\g/ \i= \R  \oplus (\s/\i)$. Consider the case $S=\widetilde{\SL}(2, \R)$. If $\i \subset \mathfrak{sl}(2, \R)$ is nilpotent, then it has nilpotency of order $2$ when acting on $\s/ \i$ (since the dimension of this space is 2) and, hence, has the same order when acting on $\R \times (\s/\i)$. Therefore, it cannot infinitesimally preserve a Lorentz product. It follows that 
 $I$ must be hyperbolic or elliptic. In both cases, the factor $\R$ is exactly the trivial factor 
 of the $I$-action. Thus, it is orthogonal to $\s/ \i$ (observe that $\R$ is not null).\\\\
 One concludes that $G/I$ is a direct product $\R \times (S/ I)$. The metric on $\R$ can be positive or negative.
 As for $S/I$, it can be the sphere $\mathbb S^2$ if $S=\SU(2)$, the hyperbolic plane $\mathbb H^2$ or the (anti) de Sitter plane $ \mathbb DS^2 =  \widetilde{\SO}_0(1, 2)/\SO_0(1, 1)$ if $S=\widetilde{\SL}(2,\R)$.  In particular, the space $G/I$ is symmetric.

 \begin{remark} In the case of $G = \R \times \SU(2)$, the unique $3$-dimensional subgroup of $G$ is $\SU(2)$. Hence, if the isotropy group $I$ is included in $\SU(2)$ then there is no Lie subgroup of $G$ transverse to it, that is, there is no $3$-dimensional Lie subgroup of $G$ having an open orbit. Whereas in the case of $G = \R \times \widetilde{\SL}(2, \R)$ and $I\subset \widetilde{\SL}(2, \R) $ hyperbolic, the subgroup $L = \R \times \Aff$, where $\Aff$ is the group of upper triangular matrices,  has an open orbit in $G/I$ but it does not act transitively.
\end{remark} 

 \section{Homogeneous locally symmetric spaces, proof of Theorem \ref{locally_symmetric}.}\label{proof-loc-sym}
\subsection{} Let $M$ be modeled on $X$ and homogeneous. Let us begin with assuming that $M$ is simply connected.  Then it has a developing map $d: M \to X$, equivariant with respect to a local isomorphism  $\psi: G \to L \subset \Isom(X)$ (that is, $d\circ g=\psi(g)\circ d$ for every $g\in G$). The $d$-image is an open  $L$-orbit. Conversely, any open orbit of a subgroup of $\Isom(X)$ is a homogeneous space modeled on $X$.\\\\  
Therefore, the proof of Theorem \ref{locally_symmetric} reduces to the study of $L$-orbits, for $L$ a subgroup of dimension $\geq 4$, essentially $\dim L = 4$, since the $5$-dimensional case follows. Observe first that in the case where  $X $ is symmetric, but not of constant curvature,
 $\dim (\Isom(X)) = 4$
 and there is nothing to prove, and  hence, we will assume that $X$ has constant curvature.

\subsection{} \label{curve} The stabilizer of a line $\R v \in T_xX$ in $\SO(T_x X, g_x) = \SO(1, 2)$ has dimension $2$ if $v$ is isotropic, and dimension $1$ otherwise. It follows that a curve can be preserved by a subgroup $\subset \Isom(X)$ of dimension at most $3$. Hence, if $\dim (L) = 4$ then any  $L$-orbit is either a surface or is open.\\ 
Similarly, the stabilizer of a plane $F  \subset T_xX$ has dimension $1$ unless $F$ is lightlike. It follows that if a connected surface $\Sigma$ is preserved by a subgroup $L \subset \Isom(X)$  of dimension $4$, then it is lightlike and homogeneous. 
\subsection{Action of $\Aff \subset \SO(1, 2)$ near $0$ in $\R^{1, 2}$} Write the metric of $\R^{1, 2} $ as $2 du dv + dx^2$. The subgroup $H \subset \SO(1, 2)$ preserving the direction $\frac{ \partial}{\partial u} = (1, 0, 0)$ is isomorphic to the affine group $\Aff$. In fact any subgroup
 in $\SO(1, 2) $ of dimension $2$ is conjugate to $H$ and equals the stabilizer of an isotropic direction.\\\\
Let $\Sigma_0 $ be the lightlike  affine $2$-plane tangent to  $\frac{ \partial}{\partial u}^\perp=  \R \frac{ \partial}{\partial u} \oplus \R \frac{ \partial}{\partial x} $ at $0$.  So, $\Sigma_0 = \left\{ v = 0\right\}$.  The $H$-orbits of the $\Sigma_0$-points are $\left\{0\right\}$, $\R(\frac{ \partial}{\partial u})\setminus\left\{0\right\}$, and (null) affine lines parallel to $\frac{ \partial}{\partial u}$ defined explicitly by $\left\{ v, = 0, x = a \right\}$ (for $a \in \R$). A level $E^-_r$ defined by $$ 2 u v + x^2 = - r^2, \ \ r \neq 0$$ is homothetic to a hyperbolic plane, and $H$ acts on it transitively. Similarly, a level $E^+_r$ given by $$ 2 u v + x^2 = + r^2, \ \ r \neq 0$$ corresponds to a de Sitter plane.  It intersects $\Sigma_0 $  on the two lines $\left\{ v = 0, x = \pm r\right\}$. Then  $E^+_r$ minus these two lines is an $H$-orbit. Finally $E_0 - \Sigma_0$ is a $H$-orbit. This is the light cone with the affine line tangent to $ \frac{ \partial}{\partial u}$ removed.\\\\
From all this, we infer that if a lightlike surface is $H$-invariant, then it is either $E_0 - \Sigma_0$ or an open subset of $\Sigma_0$. 
\subsection{$2$-dimensional $L$-orbits are lightlike and geodesic} Let $\Sigma$ be a $L$-orbit of dimension $2$. It is, thus, lightlike and has isotropy of dimension $2$.  In the Minkowski case, we can assume $0 \in \Sigma$ and its isotropy is $H$. From our previous analysis, since $\Sigma$ is lightlike and $H$-invariant, $\Sigma$ is $E_0 - \Sigma_0$ or contained in $\Sigma_0$.\\\\
Now, the subgroup of isometries preserving the cone $E_0$ is  exactly  $\SO(1, 2)$, and in particular $E_0 - \Sigma_0$ can not be preserved by a $4$-dimensional group. It follows that $\Sigma$ is open in $\Sigma_0$ and hence geodesic.  In fact, the subgroup of elements preserving $\Sigma_0$ is $H \ltimes T$, where $T \cong \R^2$ acts by translation on $\Sigma_0$. It has dimension $4$ and, thus, coincides with $L$.  We then see that $L$ has $\Sigma_0$ as a unique $2$-dimensional orbit. It has two half-Minkowski spaces as open orbits.
\subsection{}In the non-flat case, one uses the exponential map at a point $p$ to get a similar picture as that in Minkowski. One also uses Gauss lemma to see that the metric levels at $p$ are orthogonal to the geodesic rays through $p$ and hence are not lightlike, unless in the light cone. Similarly, one concludes that $L$ is the stabilizer of a lightlike geodesic surface $\Sigma$ in $X$.\\\\
The de Sitter space $\operatorname{dS}_3$ has a model $2  uv + x^2 + y^2=1$ in $\R^4$.  A lightlike geodesic surface is given by $\left\{ (u, 0, x, y) / x^2 + y^2 = +1 \right\}$. Its stabilizer $L$ is the stabilizer of the direction $\R. (1, 0, 0, 0)$ in $\SO(1, 3)$. Thus, the exterior of $\Sigma$ consists of two open orbits.\\\\
Consider now the anti de Sitter space $\operatorname{AdS}_3$. It is represented as the level $ 2 u v + x^2 - y^2 = -1$. Its intersection with a lightlike hyperplane in $\R^{2, 2}$ is up to isometry $\left\{ (u, 0, x, y) / x^2 - y^2 = -1\right\}$.  It has two connected components, and $\Sigma$ is one of them.  In other words, the stabilizer of $\Sigma$ has exactly another $2$-dimensional orbit, and has $2$ open orbits.
\subsection{Case where $M$ is not simply connected} In the flat case, the group $L$ preserving a lightlike geodesic surface is, up to conjugacy, $L= \Aff \ltimes \R^2$. If $M$, a flat homogeneous manifold with $\dim \Isom(M) = 4$, is not simply connected, then, its isometry group is a non-trivial central quotient of $L$. But $L$ has trivial center, and thus $M$ has to be simply connected.\\\\
In the case of $\operatorname{dS}_3$, the corresponding $L$ is a maximal parabolic subgroup of $\SO(1, 3)$, which is isomorphic to $(\R \times \SO(2)) \ltimes \R^2$. Here, too, $L$ has a trivial center. The same conclusion applies to $\operatorname{AdS}_3$ as $L$ in this case is isomorphic to $\Aff \times \Aff$. $\Box$
\section{Proof of Theorem \ref{Theorem 1.1}.}\label{prof 1.1}
\begin{lemma}\label{symmetric}
    Suppose that the representation $\rho:\R\to \Aut(\Heis)$ is unimodular. Then $\P_{\rho}$ is globally symmetric. Furthermore, if $\rho$ is unipotent then $\P_{\rho}$ is isometric to Minkowski, and if $\rho$ is hyperbolic (resp. elliptic) then $\P_{\rho}$ is a Cahen-Wallach space.
\end{lemma}
\begin{proof}
   First, suppose that $\rho$ is unipotent, that is (\ref{nilpotent case}) $\rho$ is given by the derivation $$\ad_T=A_{\rho}=\begin{pmatrix}
0 &0 &0 \\ 
0 &0 &1 \\
0 &0 &0
\end{pmatrix}$$
in the basis $\left\{Z,X,Y\right\}$. The isotropy in this case can be chosen to correspond to the subalgebra $\i=\R Y$. One sees that $\Span(T,X,Z)$ is an abelian subalgebra transverse to all conjugates of $\i$. This shows (see Sect. \ref{global product}) that $\P_{\rho}$ is isometric to a translation invariant metric on $\R^3$ $i.e$ to Minkowski.\\\\
Suppose now that $\rho$ is hyperbolic, given by $$A_{b}=\begin{pmatrix}
0 &0 &0 \\ 
0 &1 &0 \\
0 &0 &-1
\end{pmatrix}$$ The brackets of $\Lie(G_{\rho})$ are given by $$ [T,X]=X, \ \ [T,Y]=-Y, \ \ \text{and} \ \ [X,Y]=Z$$
And an inner Lorenz product (\ref{hyperbolic case}) on $m=\Span(T, Y^{\prime}=X-Y, Z)$ given by $$ g(Y^{\prime},Y^{\prime})=1, \ \ g(T,Z)=-\frac{1}{2}$$ and $$g(T,T)=g(Z,Z)=g(Y^{\prime},Z)=g(T,Y^{\prime})=0$$
Consider the automorphism of $\Lie(G_{\rho})$ given by $$\varphi(T)=-T, \ \ \varphi(X)=Y, \ \ \varphi(Y)=X, \ \ \text{and} \ \ \varphi(Z)=-Z.$$
So, $\varphi$ fixes $\i=\R(X+Y)$ corresponding tho the isotropy and acts on $\Lie(G_{\rho}/\i)$ by $-Id$. Furthermore, $\varphi$ preserves $g$ (on $m$) and acts isometrically on $\P_{\rho}=G_{\rho}/I$ by fixing the point $I$ with derivative $-Id$. This shows that $\P_{\rho}$ is symmetric.\\\\
Similarly, if $\rho$ is elliptic (\ref{elliptic case}) given by $$A_{0}=\begin{pmatrix}
0 &0 &0 \\ 
0 &0 &-1 \\
0 &1 &0
\end{pmatrix}$$
The brackets given by  $$[T,X]=Y, \ \ [T,Y]=-X, \ \ \text{and} \ \ [X,Y]=Z$$
As before if  $m=\Span(T, Y, Z)$ (\ref{elliptic case}). Then $m$ is $\ad_X$-invariant and we have $$\ad_V(T)=-Y, \ \ \ad_V(Y)=Z, \ \ \text{and} \ \ \ad_V(Z)=0$$ with Lorentzian inner product $$g(Y,Y)=1, \ \ g(T,Z)=1$$ and $$g(T,T)=g(Z,Z)=g(Y,Z)=g(T,Y)=0$$ And if we consider the automorphism $$\varphi(T)=-T, \ \ \varphi(X)=X, \ \ \varphi(Y)=-Y, \ \ \text{and} \ \ \varphi(Z)=-Z.$$ we obtain, the same as before, that $\P_{\rho}$ is symmetric. The last two cases are known to be the only Cahen-Wallach spaces in dimension three (see for instance \cite{CW, KO}).
\end{proof}

\begin{lemma}\label{flat}
  $\P_{\rho}$ has constant curvature if and only if it is flat. In addition, if $\P_{\rho}$ is flat then it is either isometric to the Minkowski space with $G_{\rho}$ isomorphic to $\R\ltimes \R^3$ where $\R$ acts via a unipotent one-parameter subgroup of $\SO(1,2)$, or $\P_{\rho}$ is isometric to half Minkowski with $G_{\rho}$ non-unimodular isomorphic to $\Aff\ltimes \R^2$.

\end{lemma}
\begin{proof}
The first statement follows from the following fact: the (local) holonomy of a non-flat three dimensional Lorentz manifold with constant curvature is maximal. In particular, there cannot be a parallel null vector field. So if $\P_{\rho}$ has constant curvature then it must be flat.\\
Suppose now that $\P_{\rho}$ is flat. Then $\Lie(G_{\rho})$ has a faithful representation in $\so(1,2)\ltimes \R^{3}$. Consider the projection homomorphism
\begin{center}
    $\pi:\so(1,2)\ltimes \R^{3} \longrightarrow \so(1,2)$
\end{center}
We know that $\Lie(G_{\rho}) \cap \R^{3}$ is non-trivial and equals the kernel of $\pi$ restricted to $\Lie(G_{\rho})$ and also $\Lie(G_{\rho}) \cap \so(1,2)$ has dimension at least $1$. The case $\dim(\pi(\Lie(G_{\rho})))=3$ is not possible since $G_{\rho}$ is solvable. If that $\dim(\pi(\Lie(G_{\rho})))=1$ then necessarily $\pi(\Lie(G_{\rho}))=\Lie(G_{\rho}) \cap \so(1,2)$. In this case, $\R^3\subset\Lie(G_{\rho})$ and $G_{\rho}$ is isomorphic to $\R \ltimes \R^{3}$ where $\R$ must act via a unipotent one parameter group (otherwise $G_{\rho}$ doesn't contain a copy of Heisenberg which is not the case). In this case, $\P_{\rho}$ is isometric to Minkowski (see Lemma \ref{symmetric}).\\
Suppose now that $\dim(\pi(\Lie(G_{\rho})))=2$, then  $\dim(\Lie(G_{\rho})\cap \R^{3})=2$. So $\Lie(G_{\rho})$ is contained in the stabilizer of a (degenerate) plane $\R^{2}$ inside $\so(1,2)\ltimes \R^{3}$ (acting by conjugacy), which is isomorphic to $\operatorname{aff} \ltimes \R^{3}$. Denote the basis of the Lie algebra $\operatorname{aff}\ltimes \R^{3}$ by $\left\{h,f,e_{1},e_{2},e_{3}\right\}$ such that the stabilized plane is $\Span(e_1,e_2)$ and the brackets are given by

$$[h,f]=f \ \ \text{and} \ \ [e_{i},e_{j}]=0$$ and the matrices of $\ad_{h}, \ad_{f}$ acting on $\R^3$ with respect to the basis $\left\{e_{1},e_{2},e_{3}\right\}$ are given by
$$\ad_{f}=\begin{pmatrix}
0 &-1 &0 \\ 
0 &0 &1 \\
0 &0 &0
\end{pmatrix} \ \ \text{and} \ \
\ad_{h}=\begin{pmatrix}
1 &0 &0 \\ 
0 &0 &0 \\
0 &0 &-1
\end{pmatrix}$$
So we can suppose that $\Lie(G_{\rho})=\Span( h+\beta e_3, f+\alpha e_{3},e_{1},e_{2})$. We have that $[h +\beta e_{3},f+\alpha e_{3}]=f-\alpha e_{3}-\beta e_{2}$. Since the derived subalgebra must be contained in $\Span(f+\alpha e_{3},e_{1},e_{2})$, we deduce that $\alpha=0$. So 
\begin{center}
$\Lie(G_{\rho} )=\Span(h +\beta e_{3},f,e_{1},e_{2})$    
\end{center}
It is easy to see that $\Span(h +\beta e_{3},f-\beta e_{2})$ is a subalgebra transverse to the subalgebra $\Span(e_{1},e_{2})$. So $G_{\rho}$ is isomorphic to $\Aff\ltimes \R^{2}$ and it preserves half Minkowski (given by the stabilized plane for $\beta=0$) which must be isometric to $\P_{\rho}$ since this latter is unique (\ref{hyperbolic case}).
\end{proof}
This proves (4) and (5) in Theorem \ref{Theorem 1.1}. $\Box$
\begin{lemma}\label{isomorphism}
Except the two flat cases (see Lemma \ref{flat}), the following statements are equivalent
\begin{itemize}
    \item[1)] Up to scaling, $\P_{\rho}$ is locally isometric to $\P_{\rho^{'}}$.
\item[2)] $G_{\rho}$ is isomorphic  $G_{\rho^{'}}$.
\item[3)] Up to scaling, $\P_{\rho}$ is globally isometric to $\P_{\rho^{'}}$.
\end{itemize}

\end{lemma}
\begin{proof}
Let $G_{\rho}$ and $G_{\rho^{'}}$ be the associated groups. We know that their derived subgroups are exactly $\Heis$ (the only cases where this is not true are the flat ones). Suppose that $\P_{\rho}$ is locally isometric to $\P_{\rho^{'}}$. Then their Killing algebras $\operatorname{Kil}(\P_{\rho}), \operatorname{Kil}(\P_{\rho^{'}})$ are isomorphic. Taking into account that $\Lie(G_{\rho}) \subset \operatorname{Kil}(\P_{\rho})$ and $\Lie(G_{\rho^{'}}) \subset \operatorname{Kil}(\P_{\rho^{'}})$ and that the dimension of the Killing algebras is equal to $4$ (since we excluded the flat cases \ref{flat}), we deduce that $\Lie(G_{\rho}) = \operatorname{Kil}(\P_{\rho})=\operatorname{Kil}(\P_{\rho^{'}})=\Lie(G_{\rho^{'}})$ which means that $G_{\rho}$ and $G_{\rho '}$ are isomorphic. Suppose now that $G_{\rho}$ and $G_{\rho '}$ are isomorphic and put $\P_{\rho}=G_{\rho}/I$, $\P_{\rho^{'}}=G_{\rho^{'}}/I^{'}$. Let $\phi: G_{\rho} \longrightarrow G_{\rho^{'}}$ be a Lie group isomorphism. Since $\phi$ sends $\Heis$ to $\Heis$ (the derived subgroups), we have $\phi(I)\subset\Heis$ and $G_{\rho^{'}}/\phi(I)$ admits a $G_{\rho^{'}}$-invariant Lorentz metric isometric to $\P_{\rho}$. But, the uniqueness of the Lorentz metric on $\P_{\rho^{'}}$ (up to scaling and automorphisms) shows that $G_{\rho^{'}}/\phi(I)$, up to scaling, is isometric to $\P_{\rho^{'}}$. Hence $\P_{\rho}$ and $\P_{\rho^{'}}$ are isometric up to scaling. 
\end{proof}
Hence, Lemma \ref{isomorphism} together with Lemma \ref{flat} imply (6), (7), and (8) in Theorem \ref{Theorem 1.1}. $\Box$\\\\
\textit{Proof of part (3) in Theorem \ref{Theorem 1.1}.} Suppose that $\P_{\rho}$ is locally symmetric. Then $\P_{\rho}$ is locally isomorphic to a (simply connected) symmetric space $X$. Hence $\operatorname{Kil}(\P_{\rho})=\operatorname{Kil}(X)$, that is, $\P_{\rho}$ and $X$ have isomorphic Killing algebras. Assume that $\P_{\rho}$ is not flat, then (by Lemma \ref{flat}) we have $\operatorname{Kil}(\P_{\rho})=\Lie(G_{\rho})=\operatorname{Kil}(X)$. Let $\i_x\subset \operatorname{Kil}(X)$ be the (infinitesimal) isotropy subalgebra of a point $x\in X$. Then we have, since $X$ is symmetric, $\Lie(\operatorname{Hol}_x)\subset \i_x$ where $\operatorname{Hol}_x$ denotes the holonomy group at $x$. But, since $X$ is not flat, one has $\Lie(\operatorname{Hol}_x)=\i_x$. This shows that the $\operatorname{Hol}_x$-action on $T_xX$ is indecomposable because the $\ad(\i_x)$-action is indecomposable since it corresponds to the $\ad$-action of the isotropy $\i\subset \Lie(G_{\rho})$ which is nilpotent with nilpotency order equals $3$ (Lemma \ref{lem}). Therefore, $X$ is indecomposable. Furthermore, $X$ admits, locally, a parallel null vector field (coming from $\P_{\rho}$) which extends, uniquely, to a global null parallel vector field. Thus, $X$ is in fact a Cahen-Wallach space. But in dimension three, we have only two such spaces corresponding to the elliptic and hyperbolic plane waves with unimodular isometry groups $G_{\rho}$ \cite{KO}. So, from Lemma \ref{isomorphism}, $\P_{\rho}$ is a symmetric Cahen-Wallach space with $G_{\rho}$ unimodular elliptic or hyperbolic. Finally, all this shows that if $\P_{\rho}$ is locally symmetric not isometric to a Cahen-Wallach, then it is necessarily flat, and in this case (Lemma \ref{flat}) we know that we have only two possibilities corresponding to $G_{\rho}$ unimodular nilpotent which gives the Minkowski space and $G_{\rho}$ non-unimodular isomorphic to $\Aff\ltimes \R^2$ which gives half Minkowski.

Lastly, let us mention that parts (1) and (2) in Theorem \ref{Theorem 1.1} are proven in Proposition \ref{semi_direct_Heisenberg}. $\Box$

 \section{Plane waves, proof of Theorem \ref{plane_waves}.}\label{planewaves 8}
 
 \subsection{$\P_\rho$ is a plane wave.}
Remember the definition of the  Lorentz metric  on $\P_\rho$ in Proposition \ref{semi_direct_Heisenberg}. 
One sees that the Killing field $Z$ is null, and that its orthogonal distribution is generated by the Heisenberg algebra, i.e. the three vector fields $Z, X, Y$.  

Let $\l = \R Y + \R Z$. This is an abelian  $2$-dimensional Lie sub-algebra of $\g$. In a neighbourhood of  the basis point $1 I$, $Z$ and $Y$ are linearly independent (we assume that $I$ is the one parameter group generated by $X$). It follows that $\l$ generates the orthogonal distribution 
$Z^\perp$ (locally). 

It is proven in \cite{GL} (Theorem 3) that if a Lorentz $3$-manifold admits a degenerate distribution generated by two commuting Killing fields, then the Lorentz metric is a plane wave. This applies to our situation and yields that $\P_\rho$ is a plane wave in a neighbourhood of the basis point $1 I$. This obviously implies that $\P_\rho$ is everywhere a plane wave, by homogeneity. For easiness, let us give a proof of the Brinkmann property, that is, $Z$ is a parallel field on $\P_\rho$.

\begin{lemma}
    Let $(M,g)$ be a Lorentzian manifold and let $A,B,$ and $C$ be three Killing vector fields on $M$. Suppose that $A$ is null, $g(A,B)=0$, and $$[B,A]=0, \ \ [C,A]=\alpha A, \ \ \text{and} \ \ [C,B]=\beta A+\gamma B $$
    Then we have $\nabla_W A=0$ for every $W=A,B,$ or $C$.
\end{lemma}
\begin{proof}
    The Koszul formula for three Killing vector fields $X,Y,Z$ is the following $$2g(\nabla_X Y,Z)=g([X,Y],Z)+g([Y,Z],X)-g([Z,X],Y)$$
    By evaluations we find $$g(\nabla_AA,A)=g(\nabla_AA,B)=g(\nabla_BA,A)=g(\nabla_CA,A)=g(\nabla_BA,B)=0$$
    We have also $$2g(\nabla_BA,C)=g(-\beta A-\gamma B,A)-g(\alpha A,B)=0$$ $$2g(\nabla_CA,B)=g(\alpha A,B)+g(\beta A+\gamma B,A)=0$$
This proves that $\nabla_W A=0$ for every $W=A,B,$ or $C$. In particular we have $$\nabla_{(fA+gB+hC)}A=0 \ \ \text{for every} \ \ f,g,h \in C^{\infty}(M).$$
\end{proof}
In our case, we apply the lemma for $A = Z, B= Y$ and $C = T$, and get that $Z$ is parallel. $\Box$

\bigskip

The flatness of the $Z^\perp$-leaves follows straightforwardly from the fact that $Z$ and $Y$ commute. The proof of the last property of plane waves, that is $\nabla_U R = 0$, for any $U \in Z^\perp$ is detailed  in \cite{GL}. $\Box$

\subsection{A homogeneous $3$-plane wave is isometric to some $\P_\rho$ } \label{plane_wave_P} We are now going to prove that a simply connected homogeneous $3$-plane wave $M$ is isometric to some $\P_\rho$. The Heisenberg group $\Heis$ preserves the flat metric $g_0 = 2 du dv + dx^2$. It acts by translations on $x$  and $v$, and by a linear one parameter group of unipotent transformations
 $$(v, x, u) \mapsto (v + t x - \frac{t^2}{2} u, x-tu, u)$$  

Consider now more generally  a metric of the form $g_\delta = 2 du dv + \delta(u) dx^2$ where $\delta$ is a function depending on $u$. 
It is obviously invariant by translation on $x$ and $v$.  It is also invariant under    a modification of the previous  unipotent one parameter group as follows
 $$(v, x, u) \mapsto (v + t x - \frac{t^2}{2} F_\delta(u), x-tF_\delta(u), u)$$  
 where $F_\delta(u) $ is an anti-derivative of $\frac{1}{\delta(u)}$. 
The obtained $3$-dimensional isometry group of $g_\delta$, is in fact isomorphic to 
  $\Heis$  (and acts trivially on $u$).

Remember (\ref{intro_plane_waves}) that a plane wave can be defined by having  a local Rosen coordinates where the metric has the form $g_\delta = 2 du dv + \delta(u) dx^2$. It follows that it admits an isometric infinitesimal action of the Heisenberg algebra. This action preserves individually the leaves of the orthogonal distribution $\frac{\partial} {\partial v}^\perp$.\\ 
In fact, any Killing field $U$ tangent to the $\frac{\partial} {\partial v}^\perp$-distribution belongs to $\heis$. Indeed, let $F_0$ be a leaf of $\frac{\partial} {\partial v}^\perp$, then it is (locally) affinely isomorphic to the flat  $\R^2$ endowed with a degenerate Riemannian  metric $dy^2$. Transformations preserving such a structure have the form $(z, y) \to (\lambda z + t y + a, y + b)$. If $\lambda \neq 1$, this can not be the restriction of an isometry preserving individually the leaves of $\frac{\partial} {\partial v}^\perp$, since such a transformation has a transverse linear distortion $\lambda^{-1}$. But $\lambda = 1$ exactly means that the transformation belongs to $\Heis$. Therefore, the Killing field $U$ coincides with a Killing field $W \in \heis $ on $F_0$. The Killing field $U-W$ vanishes on the hypersurface $F_0$ and hence equals $0$ (at a singularity, a non-zero Killing form has a  derivative, element of $\so(1, 2)$ and one easily sees that its $0$-set has a most dimension $1$). From all this, one infers that $\heis $ is an ideal in the Killing algebra of the plane wave.\\
It follows that if the plane wave is locally homogeneous, its Killing algebra contains at least one Killing field $T$ allowing one to pass from a $\frac{\partial} {\partial v}^\perp$-leaf to another.  Thus $\g = \R T \oplus \heis$ is an extension of the Heisenberg algebra. The associated simply connected Lie group has the form $G_\rho$, for some $\rho$, and furthermore the homogeneous plane wave $M$ is locally isometric to $\P_\rho$.\\  
We  will assume $\P_\rho$ is not flat, since this particular case is treated in Theorem \ref{locally_symmetric}. So we have  a developing map $d: M \to \P_\rho$. The developing map $d$ is equivariant with respect to a  homomorphism $h: G \to G_\rho$ where $G$ is the Lie group acting (transitively) on $M$. Since $M$ and $\P_\rho$ have same Killing algebra, and $G_\rho$ is simply connected, it follows that $h$ is an isomorphism. By equivariance, $d$ is a diffeomorphism. In other words, $M$ is identified with $\P_\rho$.

\subsection{The non-simply connected case} If $M$ is homogeneous non-simply connected, then the previous discussion shows that $\widetilde{M}=\P_{\rho}$ for some $\rho$ and $M=\Gamma\backslash \P_{\rho}$ where $\Gamma\subset G_{\rho}$ is a normal subgroup (hence central). But the center of $G_{\rho}$ is trivial unless $\rho$ is unimodular where the center is exactly the direction $Z$. And $\rho$ unimodular corresponds to either the Minkowski case, or to the two Cahen-Wallach cases (see Lemma \ref{symmetric}).

\section{Global coordinates, proof of Theorem  \ref{coordinates}.}\label{proof coor} Consider  a plane wave metric given in Brinkmann coordinates as $$g_b= 2 dudv+\frac{b x^2}{u^2}du^2+dx^2$$
on $\left\{(u, v, x), u >0\right\}$. Let us say that most computations are actually done in \cite{BO}, but only at a Killing fields level. The metric $g_b$  obviously admits a one parameter group of isometries (boosts) $\phi: (u, v, x) \to  (e^tu, e^{-t} v, x)$. These isometries act transitively on the set of $u$-levels.\\ 
Now, as any plane wave, the Heisenberg algebra acts (locally) by preserving $u$-levels and transitively on each of them. It follows that $g_b$ is locally homogeneous and, if it is non-flat, its Killing algebra equals some extension $\g_\rho$ of the Heisenberg algebra  (this will be detailed).

\subsubsection{Global action of $G_\rho$} We are now going to prove that the $\g_\rho$ infinitesimal action integrates 
to a global $G_\rho$-action. For this, we have to consider the $\heis$-action in Brinkmann coordinates.  But understanding such an action involves analyzing the transform Brinkmann to  Rosen, which  sounds  highly ``transcendental''! 

\begin{fact}
Let a  metric $g$ in   Brinkmann  form  $g= 2 dudv+ H(u) x^2 du^2+dx^2$,  defined on $(v, x), \in \R^2$, and $u \in J$, $J$ an interval.  Then $\Heis$ acts (globally, not merely locally) isometrically by preserving the $u$-levels.
\end{fact} 

\begin{proof}

Equivalence   of Brinkmann and Rosen coordinates  is done in Appendix \ref{appendix} (see \cite{Blau} for more details). 

Then, we describe the $\Heis$-action in   the Rosen form  $g_\delta= 2 du dv + \delta(u) dx^2$ as given in \ref{plane_wave_P}. Let us focus on the   $\R^2$ ($ \subset \Heis$)-action given by  $(v, x, u) \mapsto (v + a , x +b, u)$ which is isometric  and generates the horizontal foliation $\F$, whose leaves are the $u$-levels. 

All these leaves are thus $\R^2$-affine homogeneous spaces. In particular, the geodesics of $g_\delta$ tangent to $\F$ do not depend of $\delta$, that is, they are usual straight lines as in the case of the Minkowski  space corresponding to $g_1= 2 du dv + dx^2$. 

The Killing field $W$  corresponding to the translation action on the $x$-coordinate, $(v, x, u) \mapsto (v  , x +b, u)$ is thus parallel 
on each $\F$-leaf. This Killing field is however not parallel for a general $g_\delta$ (on the whole space). Let us say, there exits a map 
$u \mapsto W_u$ (depending on $\delta$), such that, on each $\F_u$ , $W$ is parallel to $W_u$ $i.e$ directed by $W_u$. 

As for  the  Brinkmann  form  $g= 2 dudv+ H(u) x^2 du^2+dx^2$, the same Killing field $W$ is parallel on each  $u$-level. In order to prove that $W$ has a (complete) global flow, it is enough to show that the $u$-levels are complete. It turns out that geodesics of these levels are still  straight lines in the (Brinkmann) $(v, x)$-coordinates (see \cite{CFS}  where geodesic equations are explained and shown to have affine solutions for horizontal initial conditions) and since $(v, x) \in \R^2$, they are complete.
 
As for the action   $(v, x, u) \mapsto (v + t x - \frac{t^2}{2} F_\delta(u), x-tF_\delta(u), u)$ of the other one parameter subgroup of $\Heis$ in Rosen coordinates, it does not have geodesic orbits, but it is affine on the $u$-levels endowed with the Rosen $(v, x)$-coordinates.  This is still affine in the Brinkmann $(v,x)$-coordinates since the geodesics are straight lines. By completeness of these $u$-levels, the flow acts globally.
\end{proof}

\subsubsection{Identification of $\rho$ as a function of $b$} Let us now show that any $\P_\rho$
(for $\rho$ non-unimodular) is isometric to a metric $g_b$, for some $b = b(\rho)$. For this, it is enough to check that the Killing algebra of $g_{b(\rho)}$ equals $\g_\rho$. Now, the computation of the Killing algebra of $g_b$ is done in \cite{BO} (Formula 3.33). It is stated there that the Heisenberg algebra $\heis = \Span(Z, X, Y)$, is extended by an element $T$, such that $[T, X] = b Y,$ $[T,  Y]= X+Y$ and $[T, Z] = Z$, in other  words $\ad_T$ equals $ A_b= \left( \begin{array}{ccc}
1 & 0 & 0 \\ 
0& 0 & 1\\ 
0 & b & 1
\end{array} \right)$ in the basis $\left\{Z, X, Y\right\}$. The associated semi-direct product is thus defined  by $\rho(t) = e^{t A}$. Let now $A$ to  be a derivation  of $\heis$, with $A(Z) \neq 0$, that is $A(Z) = aZ$, $a \neq 0$. Up to scaling (of $A$), we can assume $a= 1$, and up to adding an interior derivation $\ad_u$, $u \in \heis$, we can assume that $A$ preserves a supplementary of $Z$, say the plane $\R X \oplus \R Y$. And up to conjugacy by automorphisms, $A$ has the form $A_b$. This completes the proof of Theorem \ref{coordinates} $\qed$
\subsubsection{Hyperbolic case}  We also have an explicit global Rosen  formula for hyperbolic plane waves \ref{terminology}   $$g_\alpha = 2 du dv + u^{2 \alpha} dx^2$$ defined on $(v, x) \in \R^2,$ $u >0$ (considered, in particular, in \cite{GRT}). So $\Heis$ acts isometrically affinely in these coordinates. Observe that the linear one parameter group $(v, x, u) \mapsto (e^t v,  e^{\alpha t}x,  e^{-t} u)$, also preserves $g_\alpha$.  We know that $\Heis$ is normal in $\Isom(g_\alpha)$, and thus an extension $G_\rho$,  acts isometrically in an affine manner. So $G_\rho$ is in fact a subgroup of $\Aff(\R^3)$.  We can  show  that the representation associated to $\alpha$ is $\rho(t) = e^{t A}$, where $A$ is diagonal with entries $(1, 1- \alpha, \alpha)$.  We deduce that any hyperbolic non-unimodular plane wave $\P_\rho$ has a global Rosen model $g_\alpha$.
\subsubsection{Parabolic (non-unimodular) case} $g_b= 2 dudv+\frac{ x^2}{ 4 u^2}du^2+dx^2$
\subsubsection*{Proof that $G_{\rho}$ is isomorphic to $G_{\rho^{'}}$ if and only if $b(\rho)=b(\rho^{'})$.}
In what follows, we denote $\mathfrak{g}_{\rho}$ by $\mathfrak{g}$. Since the center $\mathfrak{z}$ of $\mathfrak{g}'=[\mathfrak{g},\mathfrak{g}]$ is equal to $[\mathfrak{g}',\mathfrak{g}']$, and  $[\mathfrak{g},\mathfrak{z}]=\mathfrak{z}$,  the adjoint representation $\ad$ of $\mathfrak{g}$ induces an action of $\mathfrak{g}/\mathfrak{g}'$ on $\mathfrak{z}$, and since this action is not trivial ($\mathfrak{z}$ is not central), there is a unique $[T]\in \mathfrak{g}/\mathfrak{g}'$ that acts as the identity. Moreover, $\ad$ induces an action  $\overline{\ad} $ of $\mathfrak{g}/\mathfrak{g}'$ on $\mathfrak{g}'/\mathfrak{z}$ (this is well-defined, thanks to the two properties mentioned at the beginning of the proof), and $\det(\overline{\ad}_{[T]})=-b$, this shows that $b$ is determined by the algebra structure of $\mathfrak{g}$.  $\Box$

\section{Proof of the completeness statement in  Theorem \ref{Causality}.}\label{s10}
Fix a  metric $g= 2 dudv+\frac{b x^2}{u^2}du^2+dx^2$,  $b \neq 0$ and let $\P_\rho$  the so defined plane wave and $G_\rho$ its isometry group (or sometimes simply $\P$ and $G$). 
\subsection{Space of null geodesics} Let $\mathcal G^0$ be the space of all  inextendible geometric (i.e. non-parametric)  isotropic geodesics of $\P_\rho$. Let $\mathcal H^0$ be the space  of  horizontal ones, i.e. those for which $u$ is constant and  are in fact orbits of the parallel field $\frac{\partial}{\partial v}$. The set of vertical i.e.  non-horizontal ones is  denoted $\mathcal V^0$, so that $\mathcal G^0 = \mathcal V^0 \cup \mathcal H^0$.

\subsubsection{Example of an incomplete (vertical) null geodesic}
 Observe that $(v, x, u) \mapsto (v, -x, u)$ is isometric and hence $\left\{ x = 0\right\}$ is geodesic. The metric on the latter plane  is $du dv $. Therefore, $t \mapsto (0, 0, t)$ is a null geodesic. Its maximal existence domain is $ ]0, + \infty[$.  Let us denote this geodesic $\gamma_0$.

\subsubsection{Vertical null geodesics are non-complete} Indeed, for such a geodesic $\gamma(t) = (v(t), x(t), u(t))$, it is known that, up to a scaling and a shift of time, one can choose $u(t) = t$, that is the coordinate $t$ plays the role of (an affine) time. The (second order) equation on $x(t)$ is linear, and $v(t)$ is an integral of it. More precisely, as detailed in \cite{CFS}, $x(t)$ satisfies, up to a constant, an equation 
$\ddot{x}(t) =  \frac{x}{t^2}$ (after identifying $u$ with a multiple $c t$). Solutions are, thus, defined on $]0, + \infty[$. Observe in fact, that this applies to all non-horizontal geodesics, null or not, in particular to all timelike geodesics.

    \subsubsection{Action of $\Isom(\P_\rho)$ on $\mathcal G^0$}
 
 \begin{proposition} The action of $G_\rho$ on $\mathcal G^0$ has exactly two orbits $\mathcal V^0$ and $\mathcal H^0$. 
 In particular, any non-horizontal null geodesic is congruent, mod $G_\rho$, to the vertical geodesic $\gamma_0$. Also, any such geodesic is  the orbit of a one parameter group of $G_\rho$.  
 
 \end{proposition}

 \begin{proof}  Let $T^0$ be the space of null tangent directions of $\P_\rho$. This is circle bundle over $\left\{(v, x, u) \in \R^2 \times \R^+ \right\}$.  The set $H^0$ of horizontal null directions gives a section, and so is connected. The set $V^0$ of non horizontal directions forms a bundle with fiber $\mathbb S^1 - 1 \operatorname{pt}  \cong \R$, and  thus  is connected, too. Spaces of geodesics $\mathcal H^0$ and  $\mathcal V^0$ are  just quotients of $H^0$ and $V^0$ by the geodesic flow, and hence are connected.
 The isotropy in $G_\rho$ is unipotent (of dimension $1$) fixing only horizontal directions. It follows that $G_\rho$ acts freely on $V^0$ and has $1$-dimensional isotropy on $H^0$. Since $\dim V^0= 4$, and $\dim G_\rho = 4$, connectedness of $G_\rho$ implies that it acts transitively (and freely) on $V^0$.  In particular $\mathcal V^0$ is a homogeneous space $G_\rho / L$, where $L$ is a one parameter group. From this follows that any vertical null geodesic in an isometric image of $\gamma_0$, and that any such geodesic is preserved by a one parameter group of isometries acting freely (and hence transitively since the dimension is one) on it. As for $\mathcal H^0$, it is a $G_\rho$ homogeneous surface homeomorphic to $\R \times \R^+$. 
 \end{proof}
 
 \section{Proof of the non-extendibility statement in  Theorem \ref{Causality}.}\label{s11}
In this section, $\rho$ is fixed once for all, and hence $\P_{\rho}$ will be simply denoted $\P$.  An  non-trivial extension of $(\P, g_b)$ consists of a Lorentz manifold $(N, g)$ together with an embedding $\phi: \P \to  N$,  such that  $\phi^* g = g_b$, but $\phi $ is not onto, so  $\P^\prime = \phi(\P)$ is an open proper subset of $N$. We say that $\P$ is inextendible if such $\phi$ does not exist, say alternatively  if any isometric embedding is bijective. It is important to  note  that this depends on the regularity of the metric $g$ on $N$. 
\subsubsection{Analytic case}\label{analytic} We can assume $N$ is simply connected, and then use that any locally defined Killing field extends everywhere. The Lie algebra $\g_\rho$ then acts on $N$. The leaves can not be all of dimension $3$, since this means they are all open, but then by connectedness $\P^\prime = N$.  One can exclude orbits of dimension one (paragraph \ref{curve}), and conclude that some orbit in the closure of $\P^\prime$ has dimension $2$. The same argument applies to the $\heis$-action and yields that all its orbits have dimension 2. Therefore, this 2-dimensional $\g_\rho$ orbit  coincides with a $\heis$-orbit.  Now, these $\heis$- leaves are affine $\R^2$ endowed with a degenerate form $d y^2$. The automorphism group of such a structure is $\Aff \ltimes \R^2$, where $\R^2$ acts by translation and $\Aff$, is the stabilizer in $\GL(2, \R)$ of the parallel direction $\frac{\partial}{\partial v}$, identified to $(1, 0 ) \in \R^2$ (more explicitly $\Aff$ consists of matrices $ \left( \begin{array}{cc}

a & b  \\ 

0& 1 

\end{array} \right)$, preserving the direction of $(1, 0)$ and acting isometrically transversely). We then deduce that the Lie algebra $\g_\rho$ embeds in the Lie algebra of $\Aff \ltimes \R^2$, but the letter corresponds to the flat metric $g_0$. Let us mention that the non-extendibility in the analytic case can also be derived from the main result of \cite{DM}.  It says that if the Killing algebra of an  analytic  Lorentz metric on a connected  $3$-manifold $N$ has an open orbit, then this manifold is locally homogeneous, i.e. its (full) Killing  algebra has one orbit.  In our case, if $\P^\prime \neq N$, then the Lie algebra of $N$ has dimension $\geq 5$, which implies (Lemma \ref{flat}) that $N$ has constant curvature, hence flat. 
 \subsubsection{Smooth complete case} It is shown  in  \cite{GGN} that there exists a geodesic $\gamma(t)$ in $\P$, $t \in ]-\infty, 1[$ and a parallel $2$-plane field $Q(t)$ along $\gamma$ such that the sectional curvature of $Q(t)$ tends to $\infty$ when $t \to 1$. This implies that $\P$ can not be embedded in a geodesically complete space. Here the metric $g$ (on $N$) is assumed to be $C^2$ (and \textbf{complete}). 
 \subsubsection{A synthetic proof} 
Let us give a  brief outline of a proof that works in the $C^2$ case, but one may expect to adapt it even  to  the $C^1$ (perhaps $C^0$) case (we hope to give more details elsewhere). The lightlike geodesic foliation $\F$ tangent to $\frac{\partial }{\partial v}^\perp$ is defined on $\P^\prime$. It has complete leaves. Let us now recall  a Lipschitz regularity of such foliations.  For this, it might be better to give an example brought out in particular from \cite{Zeg1, Zeg2}. Let $B^3$ be the unit open Euclidean ball, and consider $\mathcal L_s$ a family of disjoint affine $2$-dimensional subsets, which are relatively complete in $B^3$, to mean that $\mathcal L_s$ equals the intersection with $B^3$ of a complete affine 2-plane of $\R^3$ (so in $B^3$, $\mathcal L_s$ is as complete as possible). Let $\Sigma= \cup_s \mathcal L_s$, be the support of our disjoint family. The Lipschitz regularity fact is that the family is (locally) Lipschitz and in particular extends to a similar disjoint geodesic family on the closure $\overline{\Sigma}$. It turns out this generalizes to disjoint families of geodesic hypersurfaces in general pseudo-Riemannian manifolds (in fact manifolds with a connection), see \cite{Zeg1, Zeg2}. In our situation, the foliation $\F$ of $\P^\prime$ extends to a foliation on $\overline{\P^\prime}$. If $\tau$ is a transverse curve to $\F$, then $\P \cap \tau$ is an open set of $\tau$, and because  $ \dim  \tau = 1$,  we infer that $\overline{\P^\prime}$ is a regular open  domain in $N$, i.e. a   submanifold of codimension $0$ with regular (in fact geodesic)  boundary.\\
Let $F_0$ be a boundary leaf of $\F$. By continuity, as above (\ref{analytic}), this is an affine $\R^2 $ endowed with a degenerate Riemannian metric. In order to get a contradiction as above, we need to prove that $G_\rho$ acts on $F_0$ (preserving its structures).\\
For this goal, let $p$ a point of $\P^\prime$  close enough to  $F_0$, and  $\phi \in G_\rho$  a small isometry of $\P^\prime$, say $\phi$ is in a small neighbourhood of 1 in $G_\rho$ and denote $q = \phi(p)$. Consider a cone of directions in $T_p \P^\prime$ whose geodesics all cut $F_0$ transversely. These geodesics cut $F_0$ and also nearby leaves $F$ of $\F$ in an open subset (of these leaves). For $\gamma$ such a geodesic, $F$ a leaf of $\F$, close to $F_0$, $\gamma \cap F$ is one point, and $\phi (\gamma \cap F) = \phi (\gamma) \cap \phi (F)$. This formula allows one to define an extension of $\phi $ on some open subset of $F_0$ by $\phi (\gamma \cap F_0) = \phi (\gamma) \cap   F_0$.  Observe here that $\phi$ preserves $F_0$.\\
Now, to check that everything is coherent, that is, these extensions are well defined and they give rise to a group, or at least an algebra action, observe that since we have a true action on $\P^\prime$, the  construction, (in $\P^\prime$) does not depend on the choice of the point $p$. Therefore, the extensions to the leaf $F_0$ coincide since they are limits of actions on interior leaves in $\P^\prime$.\\
From all this, we infer that $G_\rho$ acts on $F_0$, which implies that it embeds in the automorphism  group of its structure, $\Aff \ltimes \R^2$, as discussed in \ref{analytic}, which is impossible.
\subsubsection{Conformal extensions} Observe that all plane waves (in dimension 3) are conformally flat. Indeed, in Rosen coordinates $g = 2 du dv + \delta(u) dx^2 = \delta (u) (2 dw dv + dx^2)$, where $dw =  \frac{du}{\delta(u)}$. Therefore, their conformal developing in the Einstein universe $\operatorname{Ein}_{1, 2}$ is a non-trivial conformal extension.

\section{Compact models.}\label{s12}
Suppose that $\P_{\rho}$ is not flat. Then $\Isom_0(\P_{\rho})=G_{\rho}$ (by Theorem \ref{Theorem 1.1}). Let $\Isom(\P_{\rho})$ denote the full isometry group and $I_x\subset \Isom(\P_{\rho})$ the full isotropy of a point $x\in \P_{\rho}$ with identity component $I^0_x=I_x\cap G_{\rho}$ consisting of a unipotent one-parameter group of $\SO(1,2)$ when acting on $T_x\P_{\rho}$. Furthermore, $I_x/I^0_x$ is finite since the image of $I_x$ inside $\O(1,2)$ (by the faithful derivative representation) is an algebraic group. This shows that $\Isom(\P_{\rho})/G_{\rho}$ is finite, which implies that if we have a compact $(\Isom(\P_{\rho}),\P_{\rho})$-manifold, then it  has a finite cover which is a  $(G_{\rho},\P_{\rho})$-manifold. So we restrict our attention only to compact $(G_{\rho},\P_{\rho})$-manifolds.
\begin{proposition}\label{complete geom}
    Let $M$ be a closed $(G_{\rho},\P_{\rho})$-manifold. Then $M$ is complete (as a $(G_{\rho},\P_{\rho})$-structure), that is, the developing map $D:\widetilde{M}\to \P_{\rho}$ is a diffeomorphism.
\end{proposition}
\begin{proof}
    Let $Z\in \Lie(G_{\rho})$ be the vector that generates the central direction in $\heis$ and $V$ a generator of $\Lie(I)$ where $\P_{\rho}=G_{\rho}/I$. We have that $Z$ is $\ad_V$-invariant $i.e$ $\ad_V(Z)=0$. This implies that the vector field $L_*(Z)$ on $G_{\rho}$ obtained by left translating $Z$ is $I$-right invariant. So $L_*(Z)$ projects to a $G_{\rho}$-invariant vector field on $\P_{\rho}$ which we denote by $\mathcal{Z}$. Thus, $\mathcal{Z}$ can be pulled back to a well defined vector field $\mathcal{Z}_M$ on $M$. Let $\pi^*(\mathcal{Z}_M)$ be the pull back of $\mathcal{Z}_M$ on $\widetilde{M}$ where $\pi:\widetilde{M}\to M$ is the covering projection. The developing map $D$ sends $\pi^*(\mathcal{Z}_M)$ to $\mathcal{Z}$, so it sends a $\pi^*(\mathcal{Z}_M)$-trajectory to a $\mathcal{Z}$-trajectory where $D$, restricted to each trajectory, is a local diffeomorphism into the corresponding $\mathcal{Z}$-trajectory. Since the trajectories are one-dimensional, $D$ must be injective on each $\pi^*(\mathcal{Z}_M)$-trajectory.  But, since $\pi^*(\mathcal{Z}_M)$ is complete (because $\mathcal{Z}_M$ is complete) then $D$ sends a $\pi^*(\mathcal{Z}_M)$-trajectory bijectively onto a $\mathcal{Z}$-trajectory. \\\\
    Consider now the action of the Heisenberg group $\Heis\subset G_{\rho}$ on $\P_{\rho}$. Since the stabilizer of every point $x\in \P_{\rho}$ must be contained in $\Heis$ (because $\Heis$ is normal), then the orbits of this action define a $G_{\rho}$-invariant foliation by (totally geodesics degenerate) surfaces of $\P_{\rho}$. Thus, this foliation, denoted by $\mathcal{F}$, can be pulled back to a well defined foliation $\mathcal{F}_M$ on $M$. Let $\widetilde{\mathcal{F}}_M$ be the lift of the foliation $\mathcal{F}_M$ to the universal cover $\widetilde{M}$. The vector field $\pi^*(\mathcal{Z}_M)$ is everywhere tangent to $\widetilde{\mathcal{F}}_M$, so each leaf $\widetilde{\mathcal{F}}_M(p)$ is foliated by the $\pi^*(\mathcal{Z}_M)$-trajectories and the developing $D$ map sends each leaf of $\widetilde{\mathcal{F}}_M$ equivariantly to the $\mathcal{F}$-leaves. Let now $Y\in \heis\subset \Lie(G_{\rho})$ transverse to $\R Z\oplus\R V$ inside $\heis$, and $V$, as indicated above, a generator of $\Lie(I)$. We have that $\ad_V(Y)$ is proportional to $Z$. This implies that the flow of the left-invariant vector field $L_*(Y)$ on $G_{\rho}$ preserves the surfaces $gK$ where $K$ is the subgroup of $G_{\rho}$ associated to the abelian subalgebra $\Span(V,Z)$. Hence, since $G_{\rho}/K$ is nothing but the space of all $\mathcal{Z}$-leaves on $\P_{\rho}$, we have that $Y$ generates a $G_{\rho}$-equivariant flow $\phi^t$ on the space of $\mathcal{Z}$-leaves and preserving each Heisenberg leaf. This flow, thus, can be pulled back to a complete flow (an $\R$-action), $\phi_M^t$, on the space of the $\mathcal{Z}_M$-trajectories (completeness is due to the compactness of $M$, see Remark \ref{rem_comp}). Moreover, $\phi_M^t$ can be lifted to a complete flow, $\widetilde{\phi}_M^t$, on the space of all $\pi^*(\mathcal{Z}_M)$-trajectories on $\widetilde{M}$. Now, since $D$ sends a leaf $\widetilde{\mathcal{F}}_M(p)$ to the leaf $\mathcal{F}(D(p))$ by sending a $\pi^*(\mathcal{Z}_M)$-trajectory bijectively to a $\mathcal{Z}$-trajectory, we have then, for similar reasons, that $D$ sends a $\widetilde{\phi}_M^t$-orbit to a $\phi^t$-orbit injectively (because the flow $\phi^t$ is injective since the space of $\mathcal{Z}$-trajectories inside each Heisenberg leaf is identified with $\R$ and $\phi_t$ acts on it properly freely) and surjectively (because $\widetilde{\phi}_M^t$ is complete). So, $D$ sends each $\widetilde{\mathcal{F}}_M(p)$ bijectively onto $\mathcal{F}(D(p))$.\\\\
    Finally, let $T\in \Lie(G_{\rho})$ not inside $\heis$. We have that $\ad_V(T)\in \heis$. this implies, with the same ideas, that the flow of the left-invariant vector field $L_*(T)$ preserves the cosets $gL$ where $L=\Heis$. Thus, this flow is well defined on $G_{\rho}/L$ which is the space of all $\mathcal{F}$-leaves. By similar arguments, this defines a flow on the space of all $\widetilde{\mathcal{F}}_M$-leaves on $\widetilde{M}$ which corresponds bijectively to the flow on the space of the $\mathcal{F}$-leaves on $\P_{\rho}$. This proves that $D$ is bijective, hence, a diffeomorphism.
\end{proof}
\begin{remark} \label{rem_comp}
    Observe that the space of $\mathcal{Z}_M$-trajectories, introduced in the proof, is pathological in general yet the $\R$-action of $\phi_M^t$ is well defined on it and it can also have a wild behaviour (like a transverse $\R$-action, on the leaves of a dense linear foliation on the $2$-torus, which has a dense stabilizer). In fact, one can even construct a vector field on $M$, tangent to the Heisenberg leaves and transverse to the $\mathcal{Z}_M$-orbits, and whose flow sends orbit to orbit realizing our flow $\phi_M^t$ on the space of $\mathcal{Z}_M$-orbits.
\end{remark}
\begin{corollary}\label{compact quot}
    If $M$ is a closed $(G_{\rho},\P_{\rho})$-manifold. Then $M=\Gamma\backslash \P_{\rho}$ where $\Gamma\subset G_{\rho}$ is a discrete subgroup acting freely and properly discontinuously on $\P_{\rho}$
\end{corollary}
 
\begin{remark}  \label{remark completeness}
The completeness result of the previous Proposition \ref{complete geom}
 and Corollary \ref{compact quot} is already known in the case where $\P_{\rho}$ is symmetric. It follows, in particular, from Theorem 2.1 in \cite{DZ} or from Corollary 2 in \cite{LS}.
 \end{remark}

\subsection{(Non-)existence of compact models} 

We have seen so far that if $\P_{\rho}$ admits compact models, that is, closed manifolds locally isometric to $\P_{\rho}$, then these compact models are, up to finite covers, of the form $\Gamma\backslash \P_{\rho}$, in the sense that they are quotients of $\P_{\rho}$ by a discrete subgroup acting freely and properly.

\begin{remark}\label{remark compact} As said in Remark \ref{remark_symmetric},   the symmetric case 
in Theorem \ref{no_compact} is known. Indeed, 
it is shown in \cite{KO} that non-flat Cahen-Wallach spaces do not admit compact quotients in dimension 3. So the non-existence of compact models follows form 
 the completeness result (as
mentioned in Remark \ref{remark completeness}).

We will give here a direct proof of this fact.\end{remark}

\begin{theorem}
 All Lorentz spaces $\P_{\rho}$, except two, do not support compact quotients ($i.e$ of the form $\Gamma\backslash\P_{\rho}$). The two exceptional cases correspond to left-invariant metrics on $\operatorname{SOL}$ and $\Heis$, where the latter is globally isometric to the Minkowski space.
\end{theorem} 
\begin{proof} We will prove that there is no discrete subgroup $\Gamma $ of $G_{\rho}$ acting freely and cocompactly on $\P_{\rho}$, except for two cases. this is done by considering all the possible representations $\rho:\R\to \Aut(\Heis)$ (up to equivalence) then giving  separate proofs of the (non-)existence in each case. The basis for $\Lie(G_{\rho})$ will be $\left\{Z,X,Y,T\right\}$ as usual where $Z,X,$ and $Y$ generate $\heis$ and $A_{\rho}=\ad_T$ acting on $\heis$.\\\\
$\bullet$ Let $\R$ act on $\Heis$ via the representation $\rho:\R\to \Aut(\Heis)$ generated by the derivation (\ref{hyperbolic case}) $$A_{\rho}=\begin{pmatrix}
1 &0 &0 \\ 
0 &1 &0 \\
0 &0 &0
\end{pmatrix}$$ and let $G_{\rho}=\R\ltimes_{\rho}\Heis$. We have seen (\ref{hyperbolic case}) that $G_{\rho}/I$ where $\Lie(I)=\i$ is not included in $\Span(Z,X)$ or in $\Span(Z,Y)$, is a Lorentz space. Let $\Gamma\subset G_{\rho}$ be a discrete subgroup acting freely and co-compactly on $\P_{\rho}=G_{\rho}/I$. We have that the connected Lie subgroup $[G_{\rho},G_{\rho}]$ is generated by $\Span(Z,X)$. Indeed $[G_{\rho},G_{\rho}]$ is generated by the image of $A_{\rho}$ and $Z$. Hence $[\Gamma,\Gamma]$ lies inside the abelian subgroup $H$ given by $Z$ and $X$. Clearly $\Gamma$ is not contained in $\Heis$ since $\Heis \backslash \P_{\rho} \sim \R$ is not compact. So there is $\gamma\in \Gamma$ of the form $\gamma=h\exp(tT)$ where $t\neq 0$ and $h\in \Heis$. The adjoint action of $\gamma$ on $G_{\rho}$ (the action by conjugacy) preserves $H$ and acts on it via the automorphism $\Ad(\gamma)=\Ad(h)\circ \exp(tA_{\rho})$. Since $\exp(tA_{\rho})$ restricted to $H$ is of the form $\begin{pmatrix}
e^t &0\\ 
0 &e^t \\

\end{pmatrix}$, then the action of $\Ad(\gamma)$ on $H$ is a contraction (up to forward or backward iterations). But $\Ad(\gamma)$ preserves $\Gamma\cap H$, so we necessarily have $\Gamma\cap H=\left\{1\right\}$ since $\Ad(\gamma)$ is a contraction on $H$ and $\Gamma\cap H$ is discrete. We conclude that $[\Gamma, \Gamma]=\left\{1\right\}$ which means that $\Gamma$ is abelian. Consider now the quotient $\pi:G_{\rho}\to G_{\rho}/Z$ where $Z$ refers to the center of $\Heis$. We have that $G_{\rho}/Z=\R\ltimes \R^2$ where $\R$ acts on $\R^2$ via the representation $t\to A_t=\begin{pmatrix}
e^t &0 \\ 
0 &1 \\

\end{pmatrix}$ with $\R^2$ endowed the basis $\left\{\widetilde{X}, \widetilde{Y}\right\}$ obtained by projecting $X$ and $Y$. The restriction of $\pi$ to $\Gamma$ is injective since $\Gamma\cap H=\left\{1\right\}$. Thus $\pi(\Gamma)$ is an abelian subgroup of $\R\ltimes \R^2$. Furthermore, $\pi(\Gamma)$ is not contained in the translation factor, $\R^2$, of $\R\ltimes \R^2$ because $\pi^{-1}(\R^2)=\Heis$ and $\Gamma$ is not contained in $\Heis$. So, there is an element $\alpha=A_t+v$ in $\pi(\Gamma)$ with $t\neq 0$. Suppose that $\alpha$ has a fixed point when acting on $\R^2$, then, up to conjugacy, we can assume that $A_t\in \pi(\Gamma)$. Since the line $\Span(\widetilde{Y})$ is exactly the fixed subspace of $A_t$ and $\pi(\Gamma)$ is abelian, then all elements of $\pi(\Gamma)$ preserve $\Span(\widetilde{Y})$ and, hence, they are all of the form $A_t+v$ with $v\in\Span(\widetilde{Y})$. That is $\pi(\Gamma)\subset \R\oplus\Span(\widetilde{Y})$ which implies that $\Gamma\subset \pi^{-1}(\R\oplus\Span(\widetilde{Y}))=L$ where $L$ is the Lie subgroup generated by $Z,Y$, and $T$. But, we have seen that (\ref{global product}) $\P_{\rho}$ as a homogeneous Lorentz space, is identified with a left-invariant metric on $L$. This shows that $\Gamma\backslash \P_{\rho}=\Gamma\backslash L$ and $\Gamma\subset L$ is a co-compact lattice which is impossible since $L$ is not unimodular. It remains now to treat the case where all elements of $\pi(\Gamma)$ are without fixed points when acting on $\R^2$. Suppose that we are in this situation and consider the quotient $p:\R\ltimes\R^2\to (\R\ltimes\R^2)/\Span(\widetilde{Y})$ (observe that $\Span(\widetilde{Y})$ is central in $\R\ltimes\R^2$). Then $$(\R\ltimes\R^2)/\Span(\widetilde{Y})=\R\ltimes\R=\Aff $$ and the projection $p(\pi(\Gamma))$ is an abelian subgroup which is not contained in the normal translation factor $\R$ (because $\pi^{-1}(p^{-1}(\R))=\Heis$). Since every element of $\Aff$, which is not in the normal factor, has exactly one fixed point when acting on $\R$. We can assume then, up to conjugacy, that all elements of $p(\pi(\Gamma))$ fix $0$. This implies that $p(\pi(\Gamma))\subset \R$ (where $\R$ here refers to the "linear" factor). This gives that, as before, that $\Gamma \subset L$ (where $L$ is as above) and, for the same reason, this is impossible. Observe, in this case, that $\P_{\rho}$ supports compact models ($i.e$ compact manifolds locally isometric to $\P_{\rho}$) since it is flat (Lemma \ref{flat}), but the previous discussion shows that $\P_{\rho}$ does not support $(G_{\rho}, \P_{\rho})$-compact models.\\\\
$\bullet$ Let $\R$ act on $\Heis$ via the representation $\rho:\R\to \Aut(\Heis)$ generated by the derivation (\ref{hyperbolic case}) $$A_{\rho}=\begin{pmatrix}
1+b &0 &0 \\    
0 &1 &0 \\
0 &0 &b 
\end{pmatrix}$$
where we suppose $b>0$. In this case, $\mathbb{R}$ acts on  $\Heis$ via
$$t\mapsto \begin{pmatrix}
e^{t(1+b)} &0 &0 \\    
0 &e^{t} &0 \\
0 &0 &e^{tb}
\end{pmatrix}$$
in the usual basis $\left\{Z,X,Y\right\}$. So the $\mathbb{R}$-action on $\Heis$ is a contraction (for $t\to -\infty$), which implies that $\Gamma\cap  \Heis=\left\{1\right\}$ since $\Gamma\cap  \Heis$ is preserved by the conjugacy action (necessarily contracting up to forward or backward iteration) of an element of $\Gamma$ which is not in $\Heis$. But $\Gamma\cap  \Heis$ is discrete, which implies $\Gamma\cap  \Heis=\left\{1\right\}$. In particular $\Gamma$ is abelian. Consider now the quotient $\pi:G_{\rho}\to G_{\rho}/Z$. Remark that $G_{\rho}/Z=\R\ltimes \R^2 \subset \Aff(\mathbb{R}^{2})$ where $\R$ acts on $\R^2$ via the representation 
$t\mapsto A_t=\begin{pmatrix}
e^t &0 \\ 
0 &e^{bt} \\
\end{pmatrix}$. We know that, except translations, every element of $\R\ltimes \R^{2} $ has exactly one fixed point. Moreover, we have that $\pi(\Gamma)\cap \R^2=\left\{1\right\}$, because $\pi^{-1}(\R^2)=\Heis$ and $\Gamma \cap \Heis=\left\{1\right\}$. So (up to conjugacy), all elements of $\pi(\Gamma)$ fix $0\in \mathbb{R}^{2}$, i.e $\pi(\Gamma)\subset \mathbb{R}$, where $\R$ denotes the linear factor, which implies that $\Gamma \subset \pi^{-1}(\mathbb{R})=H$ with $H$ generated by $\Span(Z,T)$. But $H\backslash \P_{\rho} \cong \mathbb{R}$ where
$H$ acts freely and properly on $\P_{\rho}$ since it doesn't intersect any conjugate of $I$. This can be seen easily, since $I\subset \Heis $ and $\Heis$ is a normal subgroup, moreover $H\cap \Heis=Z$. We conclude that $\P_{\rho}$ doesn't admit any compact quotient.\\\\
$\bullet$  Let $\R$ act on $\Heis$ via the representation $\rho:\R\to \Aut(\Heis)$ generated by the derivation (\ref{elliptic case}) $$A_{\rho}=\begin{pmatrix}
2c &0 &0 \\    
0 &c &-1 \\
0 &1 &c
\end{pmatrix}$$
with $c\neq 0$, where $\R$ acts on $Z$ by $t\mapsto e^{2tc}$ and acts by similarities on $\R X\oplus \R Y$ with non-trivial  homothety factor. A similar reasoning to the preceding case (using the same notations) shows that $\Gamma$ is abelian and that, up to conjugacy, $\Gamma \subset H$ with $H$ given by $\Span(Z,T)$
which is impossible. Hence there can't be any compact quotient of $\P_{\rho} $.\\\\
$\bullet$ Let $\R$ act on $\Heis$ via the representation $\rho:\R\to \Aut(\Heis)$ generated by the derivation (\ref{hyperbolic case}) $$A_{\rho}^{b}=\begin{pmatrix}
1+b &0 &0 \\    
0 &1 &0 \\
0 &0 &b 
\end{pmatrix}$$ 
with $b<0$. We can suppose that $b \in [-1,0[$, because $A_{\rho}^{b} $ is similar to $A_{\rho}^{\frac{1}{b}}$. Observe that
\begin{center}
    $  G_{\rho} \ \ \text{is unimodular if, and only if} \ \ b=-1$
\end{center}
We know that (\ref{hyperbolic case}) in order for  $ \P_{\rho} =G_{\rho}/I $ to carry a Lorentz metric, it is necessary and sufficient that $I$ avoids $\Span(X,Z) \cup \Span(Y,Z)$. Let us assume, up to conjugacy, that $I$ lives inside $\Span(X,Y)$ and generated by $X+Y$. The orbits of  $\R$ action on $\Span(X,Y)$ are hyperbolas, so the orbit of $I$ under this action covers two opposite quarters among the four quarters of $\Span(X,Y)\setminus (\R X\cup \R Y)$. This shows that all the conjugates of $I$ cover two opposite quarters of the complimentary of $\Span(X,Z) \cup \Span(Y,Z)$ inside $\Heis$. Now in order for $\Gamma$ to act freely, it must not intersect any conjugate of $I$, So $\Gamma \cap \Heis$ must be contained in the other two quarters (with their boundaries included). Consider the quotient $\pi: G_{\rho}\to G_{\rho}/Z= \R\ltimes \R^2$, so $\pi(\Gamma)\cap \R^2$ must be inside some line $l\subset \R^2$ because, otherwise, the intersection contains a lattice and this would contradict that it is inside two quarters. Since $\pi(\Gamma) $ is not contained in $\R^{2}$, we can pick an element of $\alpha \in \pi(\Gamma) \setminus \R^{2}$, of the form $\alpha=\begin{pmatrix}
e^t &0 \\ 
0 &e^{bt} \\
\end{pmatrix}+v$. So the induced action of $\alpha$ on  $\R^2$ (by conjugacy)  is equivalent to the action of $\begin{pmatrix}
e^t &0 \\ 
0 &e^{bt} \\
\end{pmatrix}$. Hence either $\pi(\Gamma)\cap \R^2\subset \R X$ or $\pi(\Gamma)\cap \R^2\subset \R Y$. Suppose that $\pi(\Gamma)\cap\R^2\subset \R X$ (the case where $\pi(\Gamma)\cap\R^2\subset \R Y$ can be treated similarly) and consider the projection
\begin{center}
    $P:\R\ltimes\R^{2}\to \R\ltimes\R^{2}/\R X \cong \Aff$
\end{center}
Since $P(\pi(\Gamma))$ is abelian, and it is not contained in the normal factor, then, up to conjugacy, it is contained in the linear factor $\R\subset \Aff$. This means that up to conjugacy we have $\Gamma \subset L_{X}$ is a cocompact lattice with $L_{X}$ generated by $\Span(Z,X,T)$, but this is impossible since $L_{X}$ is not unimodular. Thus it remains (similarly) that $\Gamma \subset L_{Y}$ with $L_{Y}$ generated by $\Span(Z,Y,T)$. For $b=-1$, we have $G_{\rho}$ is unimodular but $L_Y$ is not. In fact, $L_Y$ is unimodular, if and only if $b=-1/2$ (remember $b\in [-1,0[$), and in this case, $L_Y=\operatorname{SOL}$ which admits cocompact lattices. \\\\
$\bullet$ Suppose now that (\ref{parabolic case}) $A_{\rho}=\begin{pmatrix}
2 &0 &0 \\    
0 &1 &1 \\
0 &0 &1
\end{pmatrix}$.
The $\R$-action on $\Heis$
 is given by $$t\mapsto \begin{pmatrix}
e^{2t} &0 &0 \\    
0 &e^{t} &f(t)\\
0 &0 &e^{t}
\end{pmatrix}$$
Consider the projection $\pi:G_{\rho}\to G_{\rho}/Z\cong \R\ltimes \R^2 \subset \Aff(\mathbb{R}^{2})$. For similar reasons as before we have that $\Gamma\cap \Heis=\left\{1\right\}$ since the $\R$-action is expanding on $\Heis$. So $\pi(\Gamma)$ is abelian, since $\pi(\Gamma)$ cannot be contained in $\R^2$. We conclude that  $\pi(\Gamma)$ is (up to conjugacy) contained in the linear factor, which is impossible since in this case $\Gamma$ is inside  $H$ given by $\Span(Z,T)$. \\\\
$\bullet$ Suppose now (\ref{elliptic case}) $A_{\rho}=\begin{pmatrix}
0 &0 &0 \\    
0 &0 &-1 \\
0 &1 &0
\end{pmatrix}$.
The $\R$-action fixes $Z$ and acts by rotations on $\Span(X,Y)$. We have
$I\subset \Heis\setminus Z$ and 
by hypothesis, $\Gamma$ acts freely. So $\Gamma \cap \Heis \subset Z$ because the orbit of $I$ under the conjugacy action covers $ \Heis\setminus Z$.
Consider Now the projection 
$\pi:G_{\rho}\to G_{\rho}/Z\cong \widetilde{\operatorname{Euc}}(\R^2)=\R\ltimes \R^2  $.
Then $\pi(\Gamma) \subset \widetilde{\operatorname{Euc}}(\R^2) $ is abelian (because $[\Gamma,\Gamma]\subset Z$) and it is not contained in the normal factor $\mathbb{R}^{2}$. Since $\Gamma \cap \Heis \subset Z$, we deduce that $\pi(\Gamma) \cap \mathbb{R}^{2}={0}$. Now, we know that, except the center (which is isomorphic to $\Z$), every element of $\widetilde{\operatorname{Euc}}(\R^2)$ either has exactly one fixed point or has no fixed point if it belongs $\R^2$ (the normal) which is not our case since $\pi(\Gamma)\cap\R^2=\left\{0\right\}$. Suppose that there is $\alpha \in \pi(\Gamma)$ with exactly one fixed point, then, up to conjugacy, one can suppose that all elements fix the origin, that is, $\pi(\Gamma)\subset \operatorname{Stab}(0)=\R$. So  $\Gamma \subset H$ given by $\Span(Z,T)$ which is clearly impossible. Now if all elements of $\pi(\Gamma)$ act trivially, then $\pi(\Gamma)$ is contained in the center of $\widetilde{\operatorname{Euc}}(\R^2)$, this means that $\Gamma$ is contained in the subgroup generated by $Z$ and a discrete subgroup of $\R$ generated by $T$, which is impossible.\\\\ 
$\bullet$ Lastly, let $\R$ act on $\Heis$ via the representation $\rho:\R\to \Aut(\Heis)$ generated by the derivation (\ref{nilpotent case}) $$A_{\rho}=\begin{pmatrix}
0 &1 &0 \\ 
0 &0 &1 \\
0 &0 &0
\end{pmatrix}$$
In this case, the $\R$-action on $\Heis$ is unipotent. Since all vectors in the complement of $\operatorname{Im}(A_{\rho})\oplus\R Z=\R X\oplus\R Z$ are equivalent, we can fix $I$ to be the one parameter subgroup generated by $Y$, in this case, the Lie subalgebra, isomorphic to $\heis$, generated by $Z,X,T$ is transverse to all the conjugates of $\R Y$. So $\P_{\rho}$ is isometric to a left-invariant metric on $\Heis$ which is globally isometric to the Minkowski space (Lemma \ref{symmetric}) and admits compact quotients.
\end{proof}
\begin{remark}
   Non-existence of compact models in the unimodular case (that is on spaces that are locally isometric to Cahen-Wallach spaces) follows from \cite{LS} Corollary 2, and the non-existence of compact quotients of $3$-dimensional Cahen-Wallach spaces follows from \cite{KO}.
\end{remark}
\section{Appendix: Rosen and Brinkmann coordinates} \label{appendix}

Consider a metric  $g$ written in Rosen coordinates $(v,x,u), \ u>0$ as $g=2dudv+u^{2a}dx^2$. Consider the following coordinate change $$v=\overline{v}+\frac{a}{2}\overline{u}^{-1}\overline{x}^2, \  x=\overline{u}^{-a}\overline{x}, \  u=\overline{u}$$
We have $\frac{\partial}{\partial\overline{x}}=(au^{a-1}x, u^{-a}, 0)$, 
$\frac{\partial}{\partial\overline{u}}=(-\frac{a}{2}u^{2a-2}x^2, -au^{-1}x, 1)$, and 
$\frac{\partial}{\partial\overline{v}}=(1, 0, 0)$.
One verifies that $$ g(\frac{\partial}{\partial\overline{x}},\frac{\partial}{\partial\overline{x}})=1, \ \  g(\frac{\partial}{\partial\overline{u}},\frac{\partial}{\partial\overline{u}})=(a^2-a)\overline{u}^{-2}\overline{x}^2, \ \ g(\frac{\partial}{\partial\overline{v}},\frac{\partial}{\partial\overline{v}})=0$$
And, $$g(\frac{\partial}{\partial\overline{x}}, \frac{\partial}{\partial\overline{u}})=0, \ \ g(\frac{\partial}{\partial\overline{u}}, \frac{\partial}{\partial\overline{v}})=1, \ \ g(\frac{\partial}{\partial\overline{x}}, \frac{\partial}{\partial\overline{v}})=0 $$
Therefore, we obtain that the metric $g$ in the coordinates $(\overline{v},\overline{x},\overline{u})$ has the form $$g=2d\overline{u}d\overline{v}+(a^2-a)\overline{u}^{-2}\overline{x}^2d\overline{u}^2+d\overline{x}^2$$

This correspondence Rosen-Brinkmann covers  cases of metrics of the form 
 $g=2d {u}d {v}+b {u}^{-2} {x}^2d {u}^2+d{x}^2$, with $b \geq - \frac{1}{4}$.

\bigskip

In general, for a metric   $g = 2dudv+ \delta(u)dx^2$, one  uses a coordinate change of the form 
 $v=\overline{v}+c(u)\overline{x}^2, \  x= \delta(u)^{-1/2}\overline{x}, \  u=\overline{u}$., where $c$ is a function to be determined.
 
 In order  to get  a Brinkmann form for the metric, 
 we must have
 $g(\frac{\partial}{\partial\overline{x}}, \frac{\partial}{\partial\overline{u}})=0$. This is equivalent to $4 c(u) + \delta(u)^{1/2} a(u)= 0$, where 
 $a(u)$ is the derivative of $\delta(u)^{-1/2}$. One then verifies that indeed, the new form of the metric is of Brinkmann type 
    $g=2d\overline{u}d\overline{v}+ d(\overline{u})\overline{x}^2d\overline{u}^2+d\overline{x}^2$, where 
    $d(u) = 2 c^\prime(u) + \delta(u) a(u)^2 $.

\printbibliography

\end{document}